\documentclass[10pt]{article}

\usepackage[utf8x]{inputenc}
\usepackage{amsthm,amssymb, amsmath}
\usepackage{graphicx, float, epsfig}
\usepackage[nobysame]{amsrefs}
\usepackage{tikz}

\newcommand{\norm}[1]{\left| #1 \right|}

\renewcommand{\ge}{\geqslant}

\def\e{\epsilon}
\def\T{\mathcal{T}}
\def\leb{\textup{Leb}}

\newtheorem{theorem}{Theorem}[section]
\newtheorem{corollary}[theorem]{Corollary}

\newtheorem{lemma}[theorem]{Lemma}
\newtheorem{proposition}[theorem]{Proposition}

\newtheorem{definition}[theorem]{Definition}
\newtheorem*{remark}{Remark}

\title{Word length statistics for Teichm\"uller geodesics and singularity of harmonic measure}
\author{Vaibhav Gadre \and  Joseph Maher \and Giulio Tiozzo}

\begin{document}

\maketitle

\abstract{ %
  Given a measure on the Thurston boundary of Teichm\"uller space, 
  one can pick a geodesic ray joining some
  basepoint to a randomly chosen point on the boundary.  
  Different choices of measures may yield typical geodesics with different 
  geometric properties. In
  particular, we consider two families of measures: the ones which
  belong to the Lebesgue or visual measure class, and harmonic
  measures for random walks on the mapping class group generated by a
  distribution with finite first moment in the word metric. 
  We consider the ratio between the
  word metric and the relative metric of approximating mapping class
  group elements along a geodesic ray, and prove that this ratio tends
  to infinity along almost all geodesics with respect to Lebesgue
  measure, while the limit is finite along almost all geodesics with
  respect to harmonic measure.  As a corollary, we establish
  singularity of harmonic measure. We show the same result for
  cofinite volume Fuchsian groups with cusps. As an application, we
  answer a question of Deroin-Kleptsyn-Navas about the 
  vanishing of the Lyapunov expansion exponent. }

\tableofcontents

\section{Introduction}

We start by describing an elementary example. Consider the Poincar\'e 
disc $\mathbb{H}^2$, endowed with the hyperbolic metric of constant negative curvature.
There are several ways to select a \texttt{"}typical\texttt{"}
geodesic ray based at the origin. One way is to consider the unit
tangent space at the origin, which is a circle, with the rotationally
invariant measure, which we shall call \emph{visual measure} or
\emph{Lebesgue measure}. A choice of direction with respect to this
measure determines a unique geodesic ray.  Another way to choose a
geodesic ray is to take a random process starting at the origin, for
example Brownian motion. In this case, we obtain a path which converges to 
the boundary circle almost surely, and we can pick the geodesic ray 
joining the origin to the limit point on the boundary. 
The induced probability distribution on the boundary
is called \emph{hitting measure}, or \emph{harmonic measure}.
Note that in this particular case the visual and hitting measures are
equal, as there is a unique rotationally invariant probability measure
on the circle.  However, in general the harmonic measure is not
expected to coincide with the Lebesgue measure unless in presence of
very strong homogeneity, see for example Katok \cite{katok} and
Ledrappier \cite{led}.
An alternate way to construct harmonic measures on the boundary comes 
from random walks on groups, following Furstenberg \cite{Furstenberg}. 
Indeed, we can consider a random walk on a discrete group $G$ of isometries 
of the Poincar\'e disc, and define the harmonic measure to be  
the hitting measure of the walk on the boundary circle.
This time, the harmonic measure need not be rotationally invariant, 
so the two different ways of choosing geodesics may give rise to families 
of typical geodesics with different properties. 

In this paper, we shall study geometric properties of geodesics which are
typical with respect to Lebesgue measure and compare them to the 
properties of geodesics which are typical with respect to harmonic 
measures generated by random walks on the isometry group. 

We shall focus on two main examples: nonuniform lattices in $SL(2, \mathbb{R})$, 
and mapping class groups $Mod(S)$ of surfaces of finite type.
The two cases share the important feature that the group acts on 
a geodesic metric space, and the quotient by this action is not compact but has finite volume, 
i.e. it contains a \emph{cusp}.
In fact, we shall show that typical geodesics for the visual
measure penetrate more deeply into the cusp than typical geodesics for harmonic measure.

\subsection{Fuchsian groups} 

Let $G$ be a Fuchsian group, i.e. a discrete subgroup of $SL(2, \mathbb{R})$, 
and suppose the quotient  $G \backslash \mathbb{H}^2$
has finite volume but is not compact (such a group is also called 
a \emph{nonuniform lattice} in $SL(2, \mathbb{R})$).  

In order to measure the excursion into the cusp of typical geodesics, 
we shall consider two different metrics on the group $G$. 
As $G$ is finitely generated, we can endow it with a \emph{word metric} $d_G$
with respect to a finite set of generators. 
On the other hand, the group $G$ is hyperbolic relatively to 
the parabolic subgroups, in the sense of Farb \cite{Far}. 
Thus, $G$ can be also equipped with a \emph{relative metric} $d_{rel}$,
in which any distance in a subgroup fixing a cusp has constant length
(see section \ref{section:fuchsian}; note that this metric is usually not proper). 

Given a basepoint $x_0 \in \mathbb{H}^2 $, we may identify the unit tangent space at $x_0$ with the circle $S^1
= \partial \mathbb{H}^2$ at infinity, and the measure induced on the
boundary is absolutely continuous with respect to Lebesgue measure on the unit circle. 
Furstenberg
\cite{Furstenberg} showed that the image of a random walk on $G$ in
$\mathbb{H}^2$ under the orbit map $g \to g x_0$ converges almost
surely to the boundary, defining a harmonic measure $\nu$ on $S^1$ (see section \ref{section:fuchsian}).

Let $\gamma$ be a geodesic ray from the basepoint $x_0$, and $\gamma_t$ a
point at distance $t$ from the basepoint along $\gamma$. For each time
$t$, let $h_t$ be a group element such that $h_t x_0$ is a closest
element of the $G$-orbit of $x_0$ to $\gamma_t$. 
A way to measure the penetration into the cusp of the geodesic $\gamma_t$ is to consider the 
ratio $d_G/d_{rel}$ between the word and relative metrics, since consecutive powers
of parabolic elements increase the numerator but not the denominator.
We thus define the quantity 
\[ \rho(\gamma) := \lim_{t \to \infty} \frac{d_G(1, h_t)}{d_{rel}(1, h_t)}. \]
As we shall see, this limit exists for a full measure set of geodesics
in either measure. We shall show that the limit is finite for almost
all geodesics in harmonic measure, and infinite for almost all
geodesics in visual measure.

\begin{theorem}\label{h2}
Let $G < SL_2(\mathbb{R})$ be a Fuchsian group such that the quotient
$G \backslash \mathbb{H}^2$ is a non-compact, finite area hyperbolic
orbifold. Given a geodesic ray $\gamma$ starting at the basepoint
$x_0$, let
\[ \rho(\gamma) := \lim_{t \to \infty} \frac{d_G(1, h_t)}{d_{rel}(1,
  h_t)}, \]
where $h_t x_0$ is a closest element of the $G$-orbit of $x_0$ to
$\gamma_t$. Let $\leb$ be Lebesgue measure on the circle at infinity,
and let $\nu$ be the harmonic measure determined by a random walk
generated by a probability measure on $G$ with finite first moment in the word metric, and 
whose
support generates $G$ as a semigroup. Then there is a constant $c > 0$ such that
\[ \rho(\gamma) = \left\{ \begin{array}{cl} \infty & \text{for } \leb
\text{-almost all geodesics } \gamma \\ c & \text{for } \nu
\text{-almost all geodesics } \gamma.  \end{array} \right.  \]
\end{theorem}

Recall that two measures are \emph{mutually singular} if there are
sets which have full measure with respect to one measure, and zero
measure with respect to the other. This result shows that the sets of
geodesics with differing limits for $\rho$ exhibit the mutual
singularity of the visual and harmonic measure, giving the following
corollary.

\begin{corollary} \label{cor:h2sing}
Let $G$ be a Fuchsian group $G$ which has cofinite volume, but is not cocompact,
and $\mu$ a probability distribution on $G$ with finite first moment in the word metric, 
and whose support generates $G$ as a semigroup.
Then the harmonic measure $\nu$ determined by $\mu$ is singular with respect to
Lebesgue measure on the boundary of hyperbolic plane. 
\end{corollary} 

Guivarc'h and Le Jan \cites{GL90, GL} proved the singularity result
for the special case of the congruence subgroup $\Gamma(2)$ of $PSL(2,
\mathbb{Z})$ by studying the asymptotic winding around the cusp of the
geodesic flow on $\Gamma(2) \backslash \mathbb{H}^2$. 
The statistic $\rho(\gamma)$ that we study is similar in
spirit to asymptotic winding; our formulation in ``soft'' geometric terms 
replaces the analytic approach of \cite{GL90}  and makes it 
possible, as we shall see, to generalize the result to the mapping class group.
Alternate approaches to Corollary \ref{cor:h2sing} have been given for $SL(2, \mathbb{Z})$ 
by Deroin, Kleptsyn and Navas \cite{Der-Kle-Nav} and by Blach\`ere, Ha\"issinsky and Mathieu \cite{bhm}
for the general case.


As another application of Theorem \ref{h2}, we answer a question of
Deroin-Kleptsyn-Navas \cite{Der-Kle-Nav}. For any 
finitely generated group $G$ of circle diffeomorphisms and any point $p \in S^1$,
Deroin-Kleptsyn-Navas define the \emph{Lyapunov expansion exponent} of $G$ at
$p$ as
\begin{equation} \label{lyapexp}
\lambda_{exp}(p) := \limsup_{R \to \infty} \max_{g \in B(R)} \frac{1}{R} \log |g'(p)| 
\end{equation}
where $B(R)$ is a ball of radius $R$ in $G$ with respect to a word
metric for some finite generating set. 

\begin{theorem} \label{Lyap-intro}
For a Fuchsian group which is cofinite volume but not cocompact, 
we have 
\[
\lambda_{exp}(p) = 0
\]
for almost every $p \in S^1$ with respect to Lebesgue measure. 
\end{theorem}

The theorem answers Question 3.3 in \cite{Der-Kle-Nav} in the
affirmative. The essential idea is that, given $p \in S^1$, the group 
elements realizing the maximum of the derivative in definition \eqref{lyapexp}
are the closest ones to the geodesic ray from the basepoint to $p$, and their 
derivative grows subexponentially by Theorem \ref{h2} (see section \ref{section:lyap}).

\subsection{Mapping class groups}

The observation that $\nu$-typical geodesics wind around cusps less
than Lebesgue-typical geodesics is the starting point for the main
result of the paper, namely the generalization of Theorem \ref{h2} to
mapping class groups.

Let $G = Mod(S)$ be the mapping class group of an orientable surface 
$S$ of finite type, which acts on the Teichm\"uller space
$\T(S)$ of marked hyperbolic metrics on $S$. 
The Teichm\"uller
metric on $\T(S)$ is preserved by the action of the mapping class
group, and the quotient \emph{moduli space} $\mathcal{M}(S) = G \backslash \T(S)$ has
finite volume and is not compact. 

We shall use Thurston's compactification of Teichm\"uller space, the
space of \emph{projective measured foliations} $\mathcal{PMF}$, as a
boundary for $\T(S)$. There is a natural \emph{Lebesgue measure class}
$\leb$ on $\mathcal{PMF}$ given by pulling back Lebesgue measure from
the charts defined using train track coordinates.  The space of unit
area quadratic differentials is the (co)-tangent space to
Teichm\"uller space, and the unit cotangent space at each point may be
identified with $\mathcal{PMF}$. There is an invariant measure for the
geodesic flow known as holonomy measure, and the conditional measure
on unit tangent spheres induced by this measure is absolutely
continuous with respect to Lebesgue measure.
Kaimanovich and Masur \cite{KM} showed that
if $\mu$ is a probability distribution on $G$, whose support generates
a non-elementary subgroup, then the image of a random walk on $G$
under the orbit map $g \mapsto g X_0$ converges to a point in
$\mathcal{PMF}$ almost surely. We let $\nu$ be the corresponding
hitting measure.

In general, a geodesic ray need not converge to a unique point in
$\mathcal{PMF}$, see for example Lenzhen \cite{lenzhen}. However, for
each uniquely ergodic foliation in $\mathcal{PMF}$ there is a unique
geodesic ray from any point in $\T(S)$ which converges to that
foliation.  The set of uniquely ergodic foliations has full measure
with respect to both measures $\leb$ and $\nu$, and so with respect to
either measure we may identify a full measure set of points in
$\mathcal{PMF}$ with Teichm\"uller geodesic rays from a basepoint
$X_0$.

The mapping class group is finitely generated, and we shall write
$d_G$ for a choice of word metric on $G$. We shall let $d_{rel}$ be
the word metric with respect to an infinite generating set, consisting
of adding to a finite generating set the stabilizers of simple closed
curves $\alpha_i$, where the $\alpha_i$ are a set of representatives
for orbits of simple closed curves under $G$, see Masur and Minsky
\cite{mm1}. The relative metric is also quasi-isometric to distance in
the \emph{curve complex} $\mathcal{C}(S)$.
Let $\mathcal{T}_\epsilon$ be the $\epsilon$-thin part of
Teichm\"uller space, i.e. the set of surfaces which contain a simple
closed curve of hyperbolic length at most $\epsilon$. In this case, we
restrict to taking limits over points $\gamma_t$ which do not lie in
the thin part $\T_\e$. The main result is the following:

\begin{theorem} \label{mcg} %
Let $G= Mod(S)$ be the mapping class group of a non-elementary surface $S$ of
finite type, and let $\T(S)$ be the Teichm\"uller space of $S$. Let
$\T_\e$ be the thin part of Teichm\"uller space, for some $\e > 0$
sufficiently small, and fix a basepoint $X_0 \notin \T_\e$. Given a
geodesic ray $\gamma$ starting at $X_0$, let
\[ \rho(\gamma) := \lim_{\stackrel{t \to \infty}{\gamma_t \not \in
    \T_\e}} \frac{d_G(1, h_t)}{d_{rel}(1, h_t)}, \]
where $h_t X_0$ is a closest element of the $G$-orbit of $X_0$ to
$\gamma_t$. Let $\leb$ be a measure on $\mathcal{PMF}$ in the Lebesgue
measure class, and let $\nu$ be the harmonic measure determined by a
random walk generated by a probability measure on
$G$ which has finite first moment in the word metric, and 
whose support generates a non-elementary subgroup of $G$ 
as a semigroup.
Then
there is a constant $c > 0$ such that
\[ \rho(\gamma) = \left\{ \begin{array}{cl} \infty & \text{for } \leb
\text{-almost all geodesics } \gamma \\ c & \text{for } \nu
\text{-almost all geodesics } \gamma.  \end{array} \right.  \]
\end{theorem}

The theorem has the following corollary for the harmonic measure:

\begin{theorem} \label{T:singMCG}
Let $\mu$ be a measure on the mapping class group with finite first moment in the 
word metric, and such that the semigroup generated by its support 
is a non-elementary subgroup of $Mod(S)$. Then the corresponding harmonic measure $\nu$ 
on $\mathcal{PMF}$
is singular with respect to Lebesgue measure.
\end{theorem}

The singularity of harmonic measure for random walks on $Mod(S)$ has been conjectured 
by Kaimanovich and Masur \cite{KM}. 

For general random walks on groups, this question has a long history 
(for a thorough discussion, see the introduction of Kaimanovich and Le Prince \cite{KLP}). 
In the context of lattices in Lie groups, Furstenberg \cites{fur71, fur73} first constructed random walks on discrete groups
whose hitting measure is absolutely continuous on the boundary. These examples
have finite first moment in the Riemannian metric on the Lie group, but 
do not have finite first moment in the word metric on the discrete subgroup (compare to Theorem \ref{h2}).

For the mapping class group, the corresponding question, i.e. whether it is possible to find a measure $\mu$
on $Mod(S)$ such that the hitting measure 
of the corresponding random walk
is absolutely continuous on $\mathcal{PMF}$, still appears to be open. 
As a consequence of 
Theorem \ref{T:singMCG}, such a measure $\mu$ cannot have finite first moment 
in the word metric on the mapping class group.
In \cite{Ga}, Gadre proved singularity of the harmonic measure for random 
walks on the mapping class group generated by measures $\mu$ with finite support.


For finitely supported random walks on discrete groups, 
on the other hand, the measure is expected to be singular; however, the question appears to be still open
even for the case of arbitrary cocompact lattices in $SL(2, \mathbb{R})$.


In this paper, we get the Lebesgue
measure statistics by using the ergodicity of the Teichm\"uller
geodesic flow, combined with estimates on the volume of the thin part
of the space of quadratic differentials.  The statistics for harmonic
measure follows from linear progress in the relative metric, combined
with sublinear tracking between geodesics and sample paths. 

Several authors have considered cusp excursions of Lebesgue-typical geodesics; 
in particular, Sullivan \cite{Sul} showed that on a non-compact hyperbolic manifold 
a generic geodesic ray 
ventures into the cusps
infinitely often with maximum depth in the cusps of about $\log t$,
where $t$ is the time along the geodesic ray. 
The same approach has been then adapted to the Teichm\"uller geodesic flow by Masur \cite{MasurLog}.



Our method uses essentially only the geometry of the cusp, so it is natural
to expect it to apply
to other group actions for which the orbit space 
is a non-compact manifold of finite volume and the geodesic flow is ergodic, 
e.g. for fundamental groups of higher-dimensional hyperbolic manifolds with cusps.

In the rest of the introduction, we first consider the special case of the
action of $SL(2, \mathbb{Z})$, and summarize how to proceed in the general case.

\subsection{The case of $SL(2, \mathbb{Z})$}

For the sake of exposition, we now describe in detail the case
of $SL(2, \mathbb{Z})$. This example can be described concretely 
in terms of continued fractions, and we shall see how Theorems \ref{h2}
and \ref{mcg} generalize its essential geometric features.

The group $SL(2, \mathbb{Z})$ acts on the 
hyperbolic plane $\mathbb{H}^2$ by M\"obius transformations,
preserving the \emph{Farey triangulation} of $\mathbb{H}^2$ (drawn below in the disc model).

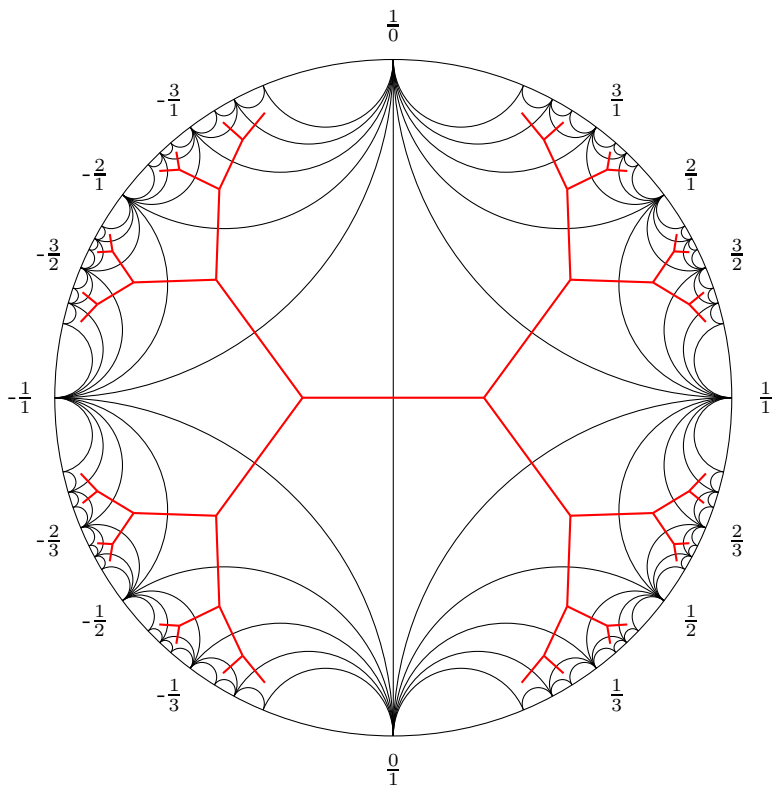
\begin{figure}[H]
\begin{center}
\begin{tikzpicture}[scale=4.5]
\input farey.tikz
\end{tikzpicture}
\end{center} \caption{The Farey triangulation in the disc model of
  $\mathbb{H}^2$.}
\end{figure}

The quotient $SL(2, \mathbb{Z})\backslash \mathbb{H}^2$ is a
hyperbolic orbifold with a cusp. It is often referred to as the
\emph{modular surface}, or the $(2, 3,\infty)$-triangle orbifold.
Given a basepoint $x_0$ in $\mathbb{H}^2$, we may
identify 
the circle at infinity $S^1 = \partial \mathbb{H}^2$ 
with the collection of geodesic rays
based at $x_0$, which is also 
also identified with the unit tangent
space at $x_0$. 
The unit tangent space has a natural measure arising from the
Riemannian metric, and we will refer to this measure as Lebesgue
measure. 

Alternatively, we may choose a geodesic by taking a random walk on the
group $SL(2, \mathbb{Z})$. By the \v{S}varc-Milnor lemma, the Cayley
graph of $SL(2, \mathbb{Z})$ is quasi-isometric to the infinite
trivalent tree that is dual to the Farey triangulation. For
simplicity, we assume that we are doing a simple random walk on this
dual tree. We may identify points on the boundary of the tree with
points on $S^1 = \partial \mathbb{H}^2$. A random walk on such a tree
converges to the boundary almost surely, and this gives a hitting
measure on $S^1$, which we shall call harmonic measure. We may then
choose a geodesic from the basepoint $x_0$ to the chosen point at
infinity.

The two measures on the boundary are in fact mutually singular, and
furthermore, we can describe sets which have full measure in one
measure, and measure zero in the other measure, in terms of the
behaviour of the geodesic rays in the modular surface.

\begin{figure}[H]
\begin{center}
\begin{tikzpicture}[scale=6]
\input geodesic.tikz
\end{tikzpicture}
\end{center} \caption{A geodesic in the Farey
  triangulation. Note that the cutting sequence starts with three right turns followed by two left turns, 
so $a_1 = 3$, and $a_2 \ge 2$, thus the endpoint of the geodesic on the circle at infinity lies 
somewhere between $\frac{1}{3}$ and $\frac{2}{7} = \frac{1}{3 + \frac{1}{2}}$. } \label{pic:farey geodesic}
\end{figure}
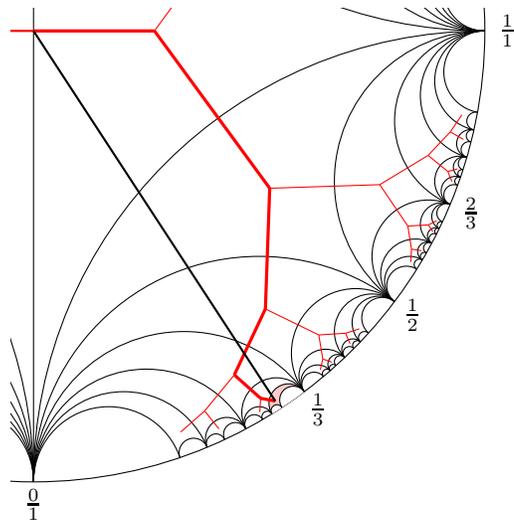

In this case, geodesic rays through the basepoint can be completely 
described in terms of continued fractions, following Series \cite{Ser}.
Indeed, a geodesic from the basepoint to a particular point $r \in S^1$
passes through some sequence of fundamental domains, or equivalently,
corresponds to a particular path in the trivalent tree converging to
$r$ at infinity (Figure \ref{pic:farey geodesic}). Starting from the basepoint, we may describe this path by a
\emph{cutting sequence}, i.e. a sequence of right and left turns depending 
on which branches of the tree the path is following. For instance, Figure \ref{pic:farey geodesic} shows 
a geodesic whose path in the tree starts off with three right turns followed by two left turns, so 
its cutting sequence starts with 
$$RRRLL\dots$$
(where $R$ stands for ``right turn'' and $L$ for ``left turn'').
The cutting sequence precisely determines the endpoint of the geodesic: indeed, 
if the geodesic ray ending at $r$ has cutting sequence
\[ \stackrel{a_0}{\overbrace{L L\dots L}} \ \stackrel{a_1}{\overbrace{RR\dots R}} \ \stackrel{a_2}{\overbrace{LL\dots L}} \ldots \]
then the continued fraction expansion of $r$ is precisely
\[ r =  a_0 + \frac{1}{a_1 + \frac{1}{a_2 + \cdots }}  \]
i.e. the $a_i$'s correspond to the number of consecutive right and
left turns along the path in the dual tree.
It is a classical result,
going back to Gauss, that for large $i$ the proportion (according to Lebesgue measure) 
of numbers with continued fraction expansions containing $a_i = n$ is about $1/n^2$. 
Since this distribution 
has infinite first moment, one gets by the ergodic theorem 
the well-known fact (see e.g. Khinchin \cite{Kh}) that for Lebesgue-almost 
all $r \in \mathbb{R}$ we have 
\begin{equation} \label{infmoment}
\lim_{n \to \infty} \frac{a_1(r) + \dots + a_n(r)}{n} = + \infty 
\end{equation}
where $a_i(r)$ is the $i$-th coefficient in the continued fraction
expansion of $r$.  Consider now instead a simple nearest neighbour
random walk on the dual tree.  For simplicity, let us consider the
measure which assigns probability $1/2$ to the right turn and $1/2$
the left turn, and let $\nu$ be its harmonic measure on $\partial
\mathbb{H}^2$.  Neglecting for now the possibility of backtracking,
the left and right turns occur independently with equal probability,
so the probability that $a_i = n$ is precisely the probability of
randomly choosing to turn in the same direction for $n$ consecutive
times, i.e. $1/2^n$.  As an exponential distribution has finite first
moment, then we have that for $\nu$-almost every $r$ the ratio in
equation \eqref{infmoment} converges to a finite number.

In terms of geodesics, we say that a geodesic ray 
has a \emph{$1/n^2$-distribution} if the proportion of coefficients $a_i = n$
in the continued fraction expansion of its endpoint equals $1/n^2$, up to multiplicative constants, 
that is
$$\lim_{N \to \infty} \frac{\#\{1 \leqslant i \leqslant N \ : \ a_i = n \}}{N} \cong \frac{1}{n^2},$$
and that a geodesic ray has an \emph{exponential distribution} if the proportion of 
coefficients $a_i = n$ 
is $O(1/2^n)$. 
These two sets of geodesics exhibit the singularity of the measures: the
geodesics with a $1/n^2$-distribution have measure one with respect to
Lebesgue measure, and measure zero with respect to harmonic measure;
moreover, the geodesics with an exponential distribution have measure
zero with respect to Lebesgue measure, and measure one with respect to
harmonic measure.

We want to show an analogue of this result for Fuchsian groups, and
the mapping class groups of surfaces. In order to generalize the result, 
we shall rewrite the ergodic average of equation \eqref{infmoment} in a
coding-free way, making use of two metrics on the group.

Let $\gamma$ be a geodesic ray from the basepoint $x_0$ to a point $r$ on the boundary; a 
point $\gamma_t$ on $\gamma$ is contained in a particular fundamental domain $D$, and let 
$g_t$ be the group element corresponding to such fundamental domain, i.e. such that $g_t x_0$
is contained in $D$. As we have seen, the geodesic segment from $x_0$ to $\gamma_t$ 
determines a finite sequence of left and right turns, and each turn
corresponds to adding a fixed number of generators to the word length,
so the word length $d_G(1, g_t)$ of $g_t$ is proportional to the sum of
the first continued fraction coefficients $a_i$ of the endpoint $r$:
$$d_G(1, g_t) \cong a_1 + a_2 + \dots + a_n.$$
Moreover, there is an alternative metric $d_{rel}$ on the group $G$, which is the metric arising 
from the Farey graph by treating
each edge as having length one. This metric is not proper, as the Farey
graph is not locally finite, and is referred to as a \emph{relative metric},
as it is quasi-isometric to
the word metric arising from the following infinite generating set: a
finite generating set, together with the elements from a single
parabolic subgroup. The distance in the relative metric is proportional to the
number of changes from consecutive right turns to consecutive left
turns, or vice versa, and so it is proportional to $n$, the number
of continued fraction coefficients. Thus, if we want to generalize the 
ratio of equation \eqref{infmoment}, we may consider the ratio
\begin{equation} 
\label{aiprop}
\rho_t := \frac{d_G(1, g_t)}{d_{rel}(1, g_t)} \cong \frac{1}{n}
\sum_{i=1}^n a_i 
\end{equation}
of the two metrics.
We may therefore use the set of geodesics for which $\rho_t$ stays bounded,
and the set of geodesics for which $\rho_t$ tends to infinity, to
exhibit the mutual singularity of harmonic measure and Lebesgue
measure.

\subsection{Outline of the paper}

We shall first (Sections \ref{section:prelim}-\ref{section:hitting}) treat the mapping 
class group case in detail, and then in the last 
two sections deal with Fuchsian groups, for which the arguments are essentially 
the same and usually slightly easier.
In Section \ref{section:prelim} we present background material on Teichm\"uller theory; 
in particular, we review the curve complex and marking complex, define the concept of excursion
 and use results of Rafi 
in order to prove the coarse monotonicity in the word metric of the approximating group elements 
along Teichm\"uller geodesics.
In Section \ref{section:lebesgue} we prove the asymptotic result for the Lebesgue measure, 
i.e. the first claim in Theorem \ref{mcg}. This is done by considering the ergodic average with respect 
to the Teichm\"uller flow of an appropriate function defined on the moduli space of 
quadratic differentials (Theorem \ref{theorem:asylength}) and then relate the 
average to the growth rate of the word metric along typical geodesics. 
In Section \ref{section:hitting} we prove the second claim in Theorem \ref{mcg}, namely 
the asymptotics for harmonic measure. 

We then turn to Fuchsian groups: in Section \ref{section:fuchsian} we prove Theorem \ref{h2}, 
while in Section \ref{section:lyap} we discuss the Lyapunov expansion exponent and prove 
Theorem \ref{Lyap-intro}.

\subsection{Notation} \label{section:notation}

We shall find it convenient to occasionally use \emph{big O}
notation. We say that $f(x) = O(g(x))$ if there are constants $A$ and
$B$ such that $\norm{ f(x) } \leqslant A \norm{ g(x) }$ for all $x
\geqslant B$. In particular, $f(x) = O(1)$ means that the function
$f(x)$ is bounded. 
We will also write $f(x) \lesssim g(x)$ to mean that the inequality 
holds up to additive and multiplicative constants, i.e. there are constants $K$ and $c$ such that
\[f(x) \leqslant K g(x) + c, \]
and similarly $ f(x) \asymp g(x)$ will mean that there exist constants $K, c$ such that 
%
\[ \frac{1}{K} g(x) - c \leqslant f(x) \leqslant K g(x) + c. \]



\section{Preliminaries from Teichm\"uller theory} \label{section:prelim}

In Sections \ref{section:quadratic}--\ref{section:short} we review
some background material on quadratic differentials, subsurface
projections and short markings. In Sections \ref{section:projections}
and \ref{section:excursion} we review in detail some results of Rafi
\cite{Raf07} which relate subsurface projection distance first to the
twist parameter along a Teichm\"uller geodesic, and then to the
excursion distance along the geodesic. In Section
\ref{section:coarse}, we use results of Rafi \cite{Raf10} to show that
word length grows coarsely monotonically along Teichm\"uller
geodesics, and finally in Section \ref{section:projection}, we show
that a similar result holds for the nearest lattice points to the
geodesic, if they lie in the thick part of Teichm\"uller space.

\subsection{Quadratic differentials and Teichm\"uller discs} \label{section:quadratic}

Let $S$ be a hyperbolic surface of finite type, i.e. a surface of
finite area which may have boundary components or punctures. We say
such a surface $S$ is \emph{sporadic} if it is a sphere with at most
four punctures or boundary components, or a torus with at most one
puncture or boundary component.  We shall primarily be interested in
non-sporadic surfaces, as in the sporadic cases the Teichm\"uller
spaces are either trivial, or isometric to $\mathbb{H}^2$, and covered
by the Fuchsian case.

Let $S$ be a non-sporadic surface with no boundary components, but
which may have punctures. We will write $\mathcal{T}(S)$ for the
Teichm\"uller space of a surface $S$, or just $\mathcal{T}$ if we do
not need to explicitly refer to the surface. We shall consider
$\mathcal{T}$ together with the Teichm\"uller metric
\[ d_{\mathcal{T}}(x, y) = \tfrac{1}{2} \inf_f \log
K(f), \]
where the infimum is taken over all quasiconformal maps $f \colon x
\to y$, and $K(f)$ is the quasiconformal constant for the map $f$. The
mapping class group $G = \textup{Mod}(S)$ of the surface acts by isometries on
$\mathcal{T}$, and we shall write $\mathcal{T}_\epsilon$ for the
\emph{thin part} of Teichm\"uller space, i.e. all surfaces which
contain a curve of hyperbolic length at most $\epsilon$. We shall
write $\mathcal{M}$ for the quotient $G \backslash \mathcal{T}$, which
is known as moduli space.  The thin part of Teichm\"uller space is
mapping class group invariant, and we shall write
$\mathcal{M}_\epsilon$ for the subset of moduli space given by $G
\backslash \mathcal{T}_\epsilon$.

Let $\mathcal{Q}$ be the space of unit area quadratic differentials,
which may be identified with the unit cotangent bundle to
Teichm\"uller space \cite{hm}. We shall write $\pi$ for the projection
$\pi : \mathcal{Q} \to \mathcal{T}$ which sends a quadratic
differential to its underlying Riemann surface, and we shall write
$\mu_{\text{hol}}$ for the Masur-Veech measure, also known as the
holonomy measure, as it may be defined in terms of holonomy
coordinates. The measure $\mu_{\text{hol}}$ is mapping class group
invariant, and so gives a measure on the moduli space of unit area
quadratic differentials $\mathcal{MQ} = G \backslash \mathcal{Q}$,
which has finite volume \cites{M1, veech}.

A quadratic differential $q$ determines a flat structure on the
surface, which may be thought of as a union of polygons glued together
along parallel sides, where the vertices of the polygons may
correspond to points of cone angle $n \pi$, for $n \geqslant 1$. If $n \geqslant
2$, then the vertex corresponds to a zero of order $n-2$ for the
quadratic differential $q$, and for $n = 1$ the vertices correspond to
cone points of angle $\pi$ which are simple poles for the quadratic
differential, and correspond to the punctures of the surface. There is
an affine action of $SL(2, \mathbb{R})$ on the flat surface, which
gives rise to a new quadratic differential. The orbits of quadratic
differentials under the action of $SL(2, \mathbb{R})$ give a
foliation of $\mathcal{Q}$ by copies of $SL(2,\mathbb{R})$, and we
shall write $\widetilde D_q$ for the orbit of the quadratic
differential $q$. We shall write $D_q$ for the image of $\widetilde
D_q$ in $\mathcal{T}$, and this is called a Teichm\"uller disc, which
is geodesically embedded in $\mathcal{T}$. With the metric induced
from the Teichm\"uller metric, $D_q$ is isometric to the hyperbolic plane of
constant curvature $-4$, and it will be convenient for us to use
coordinates coming from the disc model of hyperbolic plane, with the
initial quadratic differential $q$ corresponding to the origin.

The group of rotations of $\mathbb{R}^2$ acts on flat surfaces, and
hence on $\mathcal{Q}$. In terms of quadratic differentials, rotation
by angle $\theta$ in $\mathbb{R}^2$ sends $q \mapsto e^{-2 i \theta} q$, and
this action is trivial on Teichm\"uller space $\mathcal{T}$.  It
follows from the definition that holonomy measure is invariant
under rotation, i.e.  $\mu_{\text{hol}}(U) = \mu_{\text{hol}}(e^{i
  \theta} U)$, for all $\theta$, for any subset $U \subset \mathcal{Q}$. In
particular, this means that if we consider the conditional measure
from $\mu_{\text{hol}}$ on the image of a point $q \in \mathcal{Q}$
under rotation, i.e. $\{ e^{i \theta} q : \theta \in [0, 2 \pi] \}$,
then this is precisely the invariant Haar or Lebesgue measure on the
circle.

Finally, given $X \in \mathcal{T}$, the space $\mathcal{Q}(X)$ of unit
area quadratic differentials on $X$ is the unit cotangent space at
$X$, and we can denote by $s_X$ the conditional measure induced by the
holonomy measure on $\mathcal{Q}(X)$.  The map $\mathcal{Q}(X) \to
\mathcal{PMF}$ which associates to each quadratic differential on $X$
the projective class of its vertical foliation pushes forward the
measure $s_X$ to a measure in the Lebesgue measure class, so we can
indifferently use $s_X$ and Lebesgue measure on $\mathcal{PMF}$ when
discussing sets of full measure.  For a thorough review of the
different measures on $\mathcal{T}(S)$, $\mathcal{PMF}$ and related
spaces, we refer the reader to Athreya, Bufetov, Eskin and Mirzakhani
\cite{abem}*{Section 2} and Dowdall, Duchin and Masur
\cite{ddm}*{Section 3}.

\subsection{Curve complex and subsurface projections} \label{section:curve}

In this section we review the properties we will use of two
combinatorial objects associated with a surface, namely the \emph{curve
complex} and the \emph{marking complex}.

We say a simple closed curve on a surface $S$ is \emph{essential} if
it does not bound a disc, and is not parallel to a puncture or
boundary component.  The \emph{curve complex} $\mathcal{C}(S)$ is a
finite dimensional but locally infinite simplicial complex whose
vertices are isotopy classes of essential simple closed curves on $S$,
and whose simplices consist of collections of curves which can be
realised disjointly on the surface. For the non-sporadic surfaces, the
curve complex is a non-empty connected simplicial complex. In the case
of a torus with one puncture or boundary component, or a sphere with
four punctures or boundary components, the definition above gives a
complex with no edges, so we alter the definition to connect two
vertices if their corresponding curves can be realised by curves which
intersect at most once (in the case of the once punctured torus) or at
most twice (in the case of the four punctured sphere). In the case of
the annulus the curve complex is defined to be the infinite graph with
vertices consisting of arcs connecting the two boundary components of
the annulus modulo isotopy fixing the endpoints with edges between two
arcs if they can be realized disjointly. The curve complex of the
annulus is quasi-isometric to $\mathbb{Z}$ with a quasi-isometry given
by the algebraic intersection number. We define the curve complex to
be empty for the remaining sporadic surfaces.

We say a subsurface $Y \subseteq S$ is \emph{essential} if each
boundary component is an essential simple closed curve in $S$.  Given
an essential subsurface $Y \subseteq S$, which is not a disc or a
three-punctured sphere, one can also consider $\mathcal{C}(Y)$, the
complex of curves of $Y$. There is a coarsely well-defined
\emph{subsurface projection} $\pi_Y \colon \mathcal{C}(S) \rightarrow
\mathcal{C}(Y)$ which we now describe. Choose an element in the
isotopy class of the curve $\gamma$ which has the minimal possible
number of intersections with $Y$, and then take a regular
neighbourhood of the union of the boundary of $Y$ with the
intersection of the curve $\gamma$ with $Y$, i.e. $N(\partial Y \cup
(\gamma \cap Y) )$. Choose a component of the boundary of this regular
neighbourhood to be $\pi_Y(\gamma)$. This is coarsely well defined.

To define the annular projection $\pi(\gamma)$ of a curve $\gamma$
with essential intersection with an annulus $A$ essentially one passes
to the annulus cover $\widetilde{S}$ of $S$ given by the core curve
$\alpha$ of $A$ and chooses $\pi_A(\gamma)$ to be a component of the
lift of $\gamma$ that is an arc running from one boundary component of
$\widetilde{S}$ to the other. The set of components of the lift of
$\gamma$ that satisfy this property form a finite diameter set in the
curve complex of the annulus $\widetilde{S}$ and so the projection is
coarsely well-defined. Finally the map $\pi_A$ has the property that if
$D_\alpha$ denotes the Dehn twist about $\alpha$, then
\begin{equation}\label{equivariance}
d_{\mathcal{C}(A)}(\pi_A(D_\alpha^n(\gamma)),\pi_A(\gamma)) = 2 + |n|.
\end{equation}
Thus, defining the projection this way achieves the desired property
of recording the twisting around $\alpha$. There is a natural
$\mathbb{Z}$ action on $\mathcal{C}(A)$ by Dehn twisting around the
core curve of the annular cover $\widehat{S}$. The group $\mathbb{Z}$
also has an inclusion into the mapping class group of $S$ as Dehn
twists around $\alpha$, and so it acts on $\mathcal{C}(A)$ through
this inclusion. The projection map $\pi_A$ is coarsely
equivariant with respect to the two $\mathbb{Z}$ actions. we will
often abuse notation and write $\pi_{\alpha}$ to mean the subsurface
projection to an annulus whose core curve is $\alpha$.

A \emph{marking} consists of a collection of simple closed curves
$\alpha_i$ forming a maximal simplex in the curve complex, or
equivalently, a pants decomposition of the surface, together with a
\emph{transverse curve} $\tau_i$ for each pants curve $\alpha_i$,
which is an element of the annular curve complex corresponding to
$\alpha_i$. The curves $\alpha_i$ are known as the \emph{base curves}
of the marking. We remark that the definition we give here corresponds
to the definition of a \emph{complete} marking from \cite{MM}. They
consider more general markings, in which the set of base curves does
not need to form a maximal simplex in $\mathcal{C}(S)$, and all base
curves are not required to have a transversal. However, complete
markings suffice for our purposes.

If $\alpha$ is a simple closed curve in $S$, then a \emph{clean
  transverse curve} for $\alpha$ is a simple closed curve $\beta$,
such that a regular neighbourhood of $\alpha \cup \beta$, isotoped to
have minimal intersection, is either a sphere with four boundary
components, or a torus with a single boundary component.  A
\emph{clean marking} is a marking $(\alpha_i, \tau_i)$, such that each
transverse curve $\tau_i$ is of the form $\pi_{\alpha_i}(\beta_i)$,
for some clean transverse curve $\beta_i$, which is disjoint from the
union of the other base base curves $ \cup \alpha_j$, for $j \not =
i$. A clean marking $m' = (\alpha_i, \beta_i)$ is compatible with a
marking $m = (\alpha_i, \tau_i)$, if the base of $m$ is equal to the
base of $m'$, and for each simple closed curve $\alpha_i$ in the base,
$d_{\mathcal{\alpha}_i}( \tau_i, \pi_{\alpha_i} \beta_i)$ is
minimal. There are only finitely many clean markings $m'$ compatible
with a given marking $m$.

The \emph{marking complex} $M(S)$ is a graph whose points are clean
markings, and whose edges are given by \emph{elementary moves} as
defined by Masur and Minsky \cite{MM}. These moves are called twists
and flips. In a \emph{twist}, a transverse curve $\beta_i$ is replaced
by the image of the transverse curve under a Dehn twist along its
corresponding pants curve $D_{\alpha_i}(\beta_i)$. In a \emph{flip}, a
transverse curve $\beta_i$ and its corresponding base curve $\alpha_i$
are interchanged, i.e. a new clean marking is chosen which is
compatible with the marking formed by replacing $(\alpha_i,
\pi_{\alpha_i}(\beta_i))$ with $(\beta_i, \pi_{\beta_i}(\alpha))$. The
mapping class group acts on the marking complex and the space of
orbits is finite. We will write $d_M$ for the induced metric on the
marking complex obtained by setting the length of each edge equal to
one.

The mapping class group is finitely generated, so a choice of
generating set gives rise to a word metric, in which the length of a
group element is the shortest length of any product of generators
representing the group element. Different generating sets give rise to
quasi-isometric metrics. We shall assumed we have fixed a generating
set, and we shall write $d_G$ for the word metric distance in the
mapping class group. Masur and Minsky showed that the distance $d_M$
in the marking complex is quasi-isometric to the word metric $d_G$ in
the mapping class group.

\begin{proposition}\cite{MM}*{Theorems 6.10 and 7.1} \label{mod-marking}
Fix a complete clean marking $m_0$ and a system of generators for
$\textup{Mod}(S)$. Then there exist constants $C_1, C_2$ such that for
each $g \in \textup{Mod}(S)$
$$C_1^{-1} d_G(1, g) - C_2 \leqslant d_M(m_0, g m_0) \leqslant C_1 d_G(1, g) +
C_2. $$
\end{proposition}

There is a coarsely well-defined map from the marking complex $M(S)$
to the curve complex $\mathcal{C}(S)$, which takes a marking to one of
the short curves in the marking.  In particular, for any essential
subsurface $Y \subseteq S$, this gives us a map from the marking space
$M(S)$ to $\mathcal{C}(Y)$, given by composing $\pi$ and
$\pi_Y$. Given markings $m$ and $n$, denote by $d_Y(m, n)$ the
diameter in $\mathcal{C}(Y)$ of the union of the projections of $m$ and $n$.  If
$\alpha$ is a simple closed curve, then $d_\alpha$ will denote the
distance in the curve complex of the annulus with core curve
$\alpha$.

Masur and Minsky \cite{MM}*{Theorem 6.12} proved a distance formula
expressing the distance in the marking complex $M(S)$, and hence by
Proposition \ref{mod-marking}, the distance in $\textup{Mod}(S)$ in
the word metric, in terms of subsurface projections. We now describe
their formula, using the cutoff function $\lfloor x \rfloor_A$,
defined by
\begin{equation} \label{floor}
\lfloor x \rfloor_A = \left\{ \begin{array}{cc} x & \text{if}
    \hskip 5pt x \geqslant A \\ 0 & \text{otherwise.} \end{array}
\right.  
\end{equation}

\begin{theorem}[\cite{MM} Quasi-distance formula] \label{Quasi-Distance} %
There exists a constant $A_0 >0$, which depends only on the topology
of the surface $S$, such that for any $A \geqslant A_0$, there are constants
$C_1$ and $C_2$, which depend only on $A$ and the topology of $S$,
such that for any pair of clean markings $m$ and $m'$ in $M(S)$,
\begin{equation*}\label{eq:quasi-distance}
C_1^{-1} d_M(m,m') - C_2 \leqslant  \sum_{Y \subseteq S}
 \lfloor d_Y(m,m') \rfloor_A \leqslant  C_1 d_M(m, m') + C_2
\end{equation*}
where the sum runs over all subsurfaces $Y$ of $S$, including $S$.
\end{theorem}

\subsection{Short curves and short markings} \label{section:short}

Given a hyperbolic surface $X$, there is a \emph{systole} map from
Teichm\"uller space $\mathcal{T}(S)$ to the curve complex
$\mathcal{C}(S)$ given by sending $X$ to a shortest curve on $X$. This
map is coarsely well defined: there may be multiple shortest curves,
but there are only finitely many choices, and they are a bounded
distance apart in the curve complex, where these bounds depend only on
the topology of $S$. This follows from the fact that by Bers' Lemma,
for any surface $S$ there is a constant $L$ depending on $S$ such that
any hyperbolic metric on $S$ contains a simple closed curve of length
at most $L$, and the collar lemma says that for any simple closed
curve $\gamma$ of length $L$, there is an $\e > 0$, depending on $L$,
such that an $\e$-neighbourhood of $\gamma$ is embedded, and so this
bounds the number of intersections of any pair of curves of length
$L$. In particular, for any Teichm\"uller geodesic $\gamma_t$, this
gives a sequence of simple closed curves $\alpha_t$.

A \emph{reparameterization} of $\mathbb{R}$ is a continuous,
monotonically increasing function $\phi \colon \mathbb{R} \to
\mathbb{R}$, which need not be onto. We say a function $f$ from
$\mathbb{R}$ to a metric space is an \emph{unparameterized $(K,
  c)$-quasigeodesic} if there is a reparameterization $\phi$ such that
$f \circ \phi$ is a $(K, c)$-quasigeodesic, which may be of finite
length.

Masur and Minsky \cite{MM} showed that the image of a Teichm\"uller
geodesic under the shortest curve map is an unparameterized
quasigeodesic in the curve complex. Rafi \cite{Raf10} showed that the
composition of subsurface projection with the shortest curve map gives
an unparameterized quasigeodesic in the curve complex of the
subsurface.

\begin{theorem}[\cite{Raf10}*{Theorem B}] \label{unparameterized}
There are constants $K$ and $c$, which only depend on the surface $S$,
such that for any Teichm\"uller geodesic $\gamma$, and any subsurface
$Y \subseteq S$, the sequence of curves $\pi_Y(\alpha_t)$ arising from
the projection of the shortest curves $\alpha_t$ to $\mathcal{C}(Y)$
is an unparameterized $(K, c)$-quasigeodesic in $\mathcal{C}(Y)$.
\end{theorem}

Given a hyperbolic surface $X$, let us define the \emph{shortest
  marking} $m(X)$ in the following way. First, choose a pants
decomposition by picking the shortest simple closed curves in the
hyperbolic metric, using the greedy algorithm. To be precise, start by
choosing one of the shortest curves on the surface, then choose one
of the shortest curves on the complementary surface, and continue
until you have a pants decomposition of the original surface.  Then,
for each curve $\alpha_i$ of the pants decomposition choose a
transverse curve $\tau_i$ which is perpendicular to $\alpha_i$ in the
hyperbolic metric. If there are multiple shortest curves, then the
shortest marking may not be unique, but there are only a finite number
of choices, with a bound depending on the topology of the surface.
This gives a map from $\T(S)$ to $M(S)$, which is coarsely
well-defined, and we shall write $m_t$ for the image of a point on a
Teichm\"uller geodesic $\gamma_t$ under this map.

\subsection{Projections and isolated intervals} \label{section:projections}

Rafi \cite{Raf07} shows that for any Teichm\"uller geodesic $\gamma$,
and any subsurface $Y$, there is a (possibly empty) interval $I_Y$
during which $Y$ is \emph{isolated}, i.e. the boundary components of
$Y$ are short in the hyperbolic metric. In order to make this statement precise,
let us pick a constant $\epsilon_0 > 0$ which is smaller than the Margulis 
constant. Given a Teichm\"uller geodesic $\gamma(t)$, and a simple closed curve
$\alpha$, we shall write $L_t(\alpha)$ for the length of $\alpha$ in
the hyperbolic metric $\gamma(t)$.


\begin{proposition}[\cite{Raf07}*{Corollary 3.3}] \label{disjointinterval}
Let $\epsilon_0 > 0$ be sufficiently small. Then there exists 
$\epsilon_1 \leqslant \epsilon_0$ such that, for any geodesic
in the Teichm\"uller space and any curve $\alpha$ in $S$, there exists
a connected (perhaps empty) interval $I_\alpha$ such that 
\begin{enumerate}
 \item for $t \in I_\alpha$, $L_{t}(\alpha) \leqslant \epsilon_0$;
\item for $t \notin I_\alpha$, $L_{t}(\alpha) \geqslant \epsilon_1$.
\end{enumerate}
\end{proposition}

Outside the active interval $I_\alpha$, the map from the Teichm\"uller
geodesic to the curve complex of the annulus corresponding to $\alpha$
is coarsely constant.

\begin{proposition}[\cite{Raf07}*{Proposition 3.7}] \label{noprogress}
There is a constant $K$, depending only on the topology of the surface
$S$, and the choices for the constants $\epsilon_0$ and $\epsilon_1$,
such that if $[r, s] \cap I_\alpha = \emptyset$, then
$$ d_\alpha( m_r, m_s ) \leqslant K. $$
\end{proposition}

In the next section we show that the length of the active interval for
an annulus is roughly $\log$ of the projection distance of the
endpoints of the geodesic into the subsurface.

\subsection{Excursions and twist parameter} \label{section:excursion}

The material in this section is due to Rafi \cites{Raf07,
  Raf10}. However, we need versions of his results in terms of the
excursion parameter, and we use some of the contents of the proofs,
not just the main stated results, so we write out all of the details
for the convenience of the reader.

A horoball $H$ in the hyperbolic plane is a subset of the plane which
in the Poincar\'e disc model corresponds to a Euclidean disc whose
boundary circle is tangent to the boundary at infinity.  Given a
horoball $H$ and a geodesic $\gamma$ which spends a finite amount of
time in $H$, let us define the \emph{excursion} $E(\gamma, H)$ of
$\gamma$ in $H$ as the \texttt{"}relative visual size\texttt{"} of the set of rays
which go deeper than $\gamma$ inside $H$. Namely, consider a basepoint
$X_0$ on the Teichm\"uller disc in $\T$, and let $\gamma_H$ be the
geodesic through $X_0$ which tends to the cusp of $H$, and $\gamma_T$
a geodesic through $X_0$ which is tangent to $H$. Let $\phi_0$ be the
angle between $\gamma$ and $\gamma_H$, and $\phi_{max}$ be the angle
between $\gamma_H$ and $\gamma_T$ (see Figure \ref{horoball}). Then

\begin{definition}
The \emph{excursion} of the geodesic $\gamma$ in the horoball $H$ is defined as
\begin{equation} \label{visualsize}
E(\gamma, H) := \frac{\phi_{max}}{\phi_0}.
\end{equation}

\end{definition}

It turns out that $E(\gamma, H)$ is, up to an additive error, also the
hyperbolic length of the projection of $\gamma \cap H$ to the
complement of $H$.

\begin{figure}[H]
\begin{center}
\begin{tikzpicture}[scale=4]

\clip (1, 1) rectangle (-1, -0.3);

\draw (0, 0) circle (1);
\draw (0, 0.66666) circle (0.33333);

\filldraw[black] (0,0) circle (0.01);

\draw (0, 0) -- (0, 1);
\draw (0, 0) -- (0.5, 0.86603);
\draw[thick] (0, 0) -- (0.2, 0.97980);

\path (0,-0.1) node {$X_0$};
\path (-0.27,0.33) node {$H$};

\path (0.23,0.15) node {$\phi_{\text{max}}$};
\path (0.05,0.6) node {$\phi_0$};

\path (0.2,0.75) node {$\gamma$};
\path (-0.05,0.85) node {$\gamma_H$};
\path (0.5,0.7) node {$\gamma_T$};

\draw[thick] (0,0) ++(60:0.2) arc (60:90:0.2);
\draw[thick] (0,0) ++(78:0.52) arc (78:90:0.52);

\end{tikzpicture}
\end{center} \caption{Excursion in the horoball $H$.}
\label{horoball}
\end{figure}
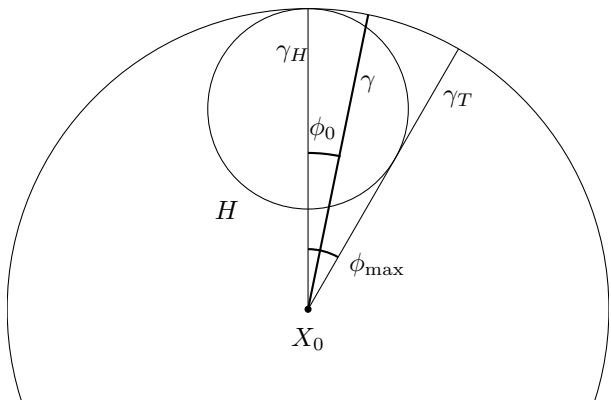

Let $(X, q)$ be a quadratic differential on $X$, and $\alpha$ a simple
closed curve on $X$.  The choice of $q$ determines a Teichm\"uller
geodesic $\gamma$ and a pair $(F^+, F^-)$ of contracting and
expanding foliations.  Each $t$ determines a new quadratic
differential $q_t$ and hence a flat metric on $X$, which we will call
the $q_t$-metric.

For a given $t$, $\alpha$ is realized by a family of parallel flat
geodesics, and we will denote as $\beta_t$ the perpendicular to
$\alpha$ in the $q_t$-metric.  The \emph{twist parameter}
$tw^+_t(\alpha)$ is the highest intersection number between a leaf of
$F^+$ and the transversal $\beta_t$, and similarly we define
$tw^-_t(\alpha)$.

Given a simple closed curve $\alpha$ corresponding to metric cylinder,
there is a unique rotation $e^{i \theta_\alpha}$ which takes the
metric cylinder to a vertical metric cylinder. The endpoint of the
geodesic ray corresponding to the quadratic differential $e^{i \theta_\alpha} q$ 
determines a point $\xi_\alpha$ on the boundary at infinity of the Teichm\"uller disc $\mathbb{D}$. 
We shall write $H_\epsilon(\alpha)$ as the set of points in the disc for which  
$\alpha$ is short in the flat metric:
$$H_\epsilon(\alpha) := \{ q \in \mathbb{D} \ : \ell_q^2(\alpha) \leqslant \epsilon \}.$$
As seen in the disc, this set is a horoball tangent to the
boundary at infinity at $\xi_\alpha$.  The fundamental estimate is the
following:

\begin{proposition} \label{twist}
Let $H = H_\epsilon(\alpha)$ as above, and let $t_1$ and $t_2$ respectively be the entry time and 
exit time from $H$ (i.e. $t_1 \leqslant t_2$) along the Teichm\"uller geodesic $\gamma$. Let moreover
$A$ be the area of the maximal flat cylinder in $(X, q_0)$ with core curve $\alpha$. 
Then we have, up to universal multiplicative and additive constants,
$$tw^-_{t_2}(\alpha) - tw^-_{t_1}(\alpha) \asymp \frac{A}{\epsilon} E(\gamma, H).$$
\end{proposition}

\begin{proof}
Consider the universal cover of the flat cylinder corresponding to
$\alpha$ at time $t$, in the flat metric $q_t$. We shall assume that
the contracting foliation is vertical, and the expanding foliation is
horizontal. Let $\ell_t$ be the length of $\alpha$ at time $t$, and
let $\theta_t$ be the angle $\alpha_t$ makes with the vertical
contracting foliation, as illustrated below in Figure
\ref{intersections}. 

\begin{figure}[H]
\begin{center}
\begin{tikzpicture}[y=-1cm, scale=0.4]

\usetikzlibrary{arrows}

\draw[thick,black] (14.95,8.4) ++(-28:2) arc (-28:-78:2);
\draw[thick,black] (15.0578,15.24667) +(-117:0.76725) arc (-117:-189:0.76725);
\draw[thick,black] (9.21556,15.42222) +(0:0.75) arc (0:55:0.75);

\draw[semithick,black] (18.03333,6.53333) -- ( 21.36633,11.1120034);
\draw[semithick,black] (16.04889,7.98667) -- (19.34889,12.52);
\draw[semithick,black] (14.08256,9.3746667) -- (17.41556,13.95333);
\draw[semithick,black] (12.04922,10.8746667) -- (15.38222,15.45333);
\draw[semithick,black] (9.98222,12.38667) -- (13.28222,16.92);
\draw[semithick,black] (8.01556,13.82) -- (11.31556,18.35333);
\draw[semithick,black] (5.98222,15.28667) -- (9.28222,19.82);
\draw[semithick,black] (5.03333,15.4) -- (22.4,15.4);
\draw[semithick,black] (15.36667,5.14444) -- (15.4,20.87778);
\draw[semithick,black] (8.16667,20.64444) -- (22.36667,10.41111);
\draw[semithick,black] (4.96667,15.97778) -- (19,5.81111);

\path (16.13333,6.56667) node[text=black,anchor=base west] {$\theta_t$};
\path (13.33333,14.5) node[text=black,anchor=base west] {$\theta_t$};
\path (8.74222,16.45333) node[text=black,anchor=base west] {$\theta_t$};

\draw[semithick,black] (17.41556,13.95333) -- (17.41556,15.4);
\draw[semithick,black] (12.2,15.44889) -- (10.22889,16.9);

\draw[very thick, latex-latex, black] (6.1,14.96667) -- (15.36667,14.96667);
\draw[very thick, latex-latex, black] (9.21556,15.42222) -- (12.16889,15.37778);

\draw[very thick,latex-latex,black] (16,7.93333) -- (19.36667,12.56667) -- (21.36667,11.13333);

\path (15.93333,14.25) node[text=black,anchor=base west] {$\ell_t$};
\path (16,16.26667) node[text=black,anchor=base west] {$h_t$};
\path (17.5,14.73333) node[text=black,anchor=base west] {$v_t$};
\path (9.6,14.43333) node[text=black,anchor=base west] {$H_t$};
\path (12.63333,11.9) node[text=black,anchor=base west] {$w_t$};
\path (17.66667,9.9) node[text=black,anchor=base west] {$\beta_t$};
\path (20.63333,12.73333) node[text=black,anchor=base west] {$\alpha$};
\path (11.16222,16.9) node[text=black,anchor=base west] {$\ell_t$};
\path (9.8,16.15) node[text=black,anchor=base west] {$H'_t$};

\end{tikzpicture}%
\end{center} \caption{Estimating intersections in the flat annulus.}
\label{intersections}
\end{figure}
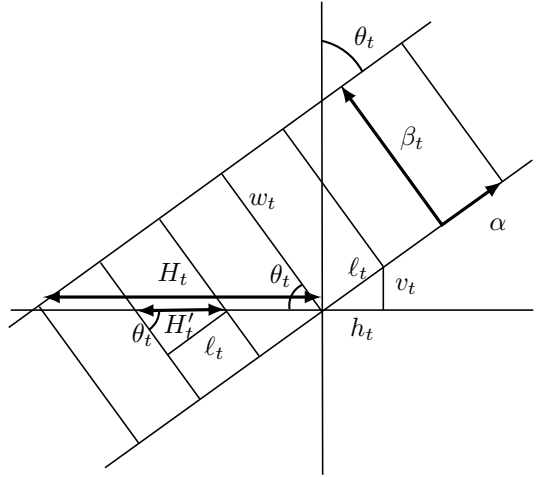

Let $h_t$ and $v_t$ be the horizontal and vertical lengths of
$\ell_t$ in the $q_t$ metric, i.e.
$$h_t = h_0 e^t = \ell_0 \sin \theta_0 e^t$$
$$v_t = v_0 e^{-t} = \ell_0 \cos \theta_0 e^{-t}.$$
Let $w_t$ be the length of $\beta_t$, which is the width of the flat
annulus. Let $H_t$ be the length of the intersection of a leaf of the
horizontal foliation with the universal cover of the flat annulus, and
let $H'_t$ be the length of the intersection of the horizontal leaf
with two adjacent translates of $\beta_t$.

Up to constant additive error, $tw^-_t(\alpha)$, which is the maximum
number of intersections between the horizontal leaf of the foliation
and $\beta$, is given by $H_t / H'_t$. Therefore
\[ tw^-_t(\alpha) = \frac{H_t}{H'_t} + O(1) =  \frac{w_t \sin
  \theta_t}{\ell_t \cos \theta_t} + O(1). \]
The area of the annulus is $A = w_t \ell_t$, and $\tan \theta_t = \tan
\theta_0 e^{2t}$, so this implies that
\begin{equation} \label{projection3}
tw^-_t(\alpha) = \frac{A }{ \ell_t^2}\tan \theta_0 e^{2t} + O(1).
\end{equation}
The total length of $\alpha$ is given by 
\begin{equation} \label{length}
\ell_t^2 = h_t^2 + v_t^2 = \ell_0^2 (\sin^2 \theta_0 e^{2t} + \cos^2 \theta_0 e^{-2t}),
\end{equation}
and recall that we choose $t_i$ such that $\ell^2_{t_i} = \epsilon$,
which by \eqref{projection3} implies
\begin{equation} \label{projectiondiff} tw^-_{t_2}(\alpha) -
  tw^-_{t_1}(\alpha) = \frac{A }{ \epsilon}\tan \theta_0(e^{2t_2}
  - e^{2t_1}) + O(1).
\end{equation}
Note that by definition the $t_i$ are solutions to the equation 
$$\ell_{t_i}^2 = \ell_0^2 (\sin^2 \theta_0 e^{2t_i} + \cos^2 \theta_0 e^{-2t_i}) = \epsilon \qquad i = 1, 2$$
If we set $X_i := e^{2 t_i}$, then $X_i$ are the solutions to 
\begin{equation} \label{eqX}
X^2 - \frac{\epsilon^2}{\ell_0^2 \sin^2 \theta_0} X + \frac{1}{\tan^2 \theta_0} = 0
\end{equation}
hence 
\begin{equation} \label{twist1}
e^{2t_2} - e^{2t_1} = X_2 - X_1 = \sqrt{\frac{\epsilon^2}{\ell_0^4 \sin^4 \theta_0} - \frac{4}{\tan^2 \theta_0} }
\end{equation}
and putting \eqref{projectiondiff} and \eqref{twist1} together
\begin{equation} \label{twist2}
tw^-_{t_2}(\alpha) - tw^-_{t_1}(\alpha) = \frac{A \tan \theta_0}{ \epsilon}(e^{2t_2} - e^{2t_1}) + O(1) = 
\frac{A }{ \epsilon}\sqrt{\frac{\epsilon^2}{\ell_0^4 \sin^2 \theta_0 \cos^2 \theta_0} - 4} + O(1). 
\end{equation}
Let us now relate this quantity to the excursion in the horoball $H$. 

\begin{lemma} \label{phimax} 
Let $\phi_{max}$ be the angle between a geodesic $\gamma_T$ tangent to
$H$ and the geodesic $\gamma_H$ which goes straight into the cusp of
$H$. Then
$$\sin \phi_{max} = \frac{\epsilon}{\ell_0^2},  $$
where $\ell_0$ is the length of $\alpha$ at time $t=0$. 
\end{lemma}

\begin{proof}
When $\theta_0 = \theta_{max}$ then the geodesic is tangent to the horoball $H$, hence 
$t_1 = t_2$ in equation \eqref{twist1}, so
$$\frac{\epsilon^2}{\ell_0^4 \sin^4 \theta_{max}} = \frac{4}{\tan^2 \theta_{max}}.$$
The claim follows by recalling that a rotation of angle $\theta$ in the flat metric picture
corresponds to multiplying the quadratic differential by $e^{2 i \theta}$, hence $\phi_{max} = 2 \theta_{max}$.
\end{proof}

The proposition now follows easily from the lemma, equation \eqref{twist2} and the fact that $\phi_0 = 2 \theta_0$:
$$tw^-_{t_2}(\alpha) - tw^-_{t_1}(\alpha) = 
\frac{2A }{ \epsilon} \frac{\sin \phi_{max}}{\sin \phi_0} 
\sqrt{1 - \frac{\sin^2 \phi_0}{\sin^2 \phi_{max} }} + O(1) \asymp \frac{A}{\epsilon} E(\gamma, H)$$ 
where in the last equality we used equation \eqref{visualsize} and the fact that $\sin \phi \asymp \phi$
(note that we can assume $\sin \phi_0 \leq \frac{1}{2} \sin \phi_{max}$, otherwise the claim 
is trivially verified).

\end{proof}

\begin{remark}
Note that one can also relate the twist parameter to the time spent by the geodesic inside the horoball, namely
$$tw^-_{t_2}(\alpha) - tw^-_{t_1}(\alpha) \asymp \frac{A}{\epsilon} (e^{t_2-t_1} - e^{t_1 - t_2}).$$
\end{remark}

The distance between the projections from the marking complex to the
complex of the annulus can be compared to the excursion in the
horoball:
\begin{proposition} \label{subsurface} Let  
$\epsilon > 0$ sufficiently small, and $(X_0, q)$ a unit area quadratic differential, which 
determines the geodesic ray $\gamma_t$. Let $A$ be the $q$-area of the maximal flat cylinder with core curve 
$\alpha$, and suppose that $\alpha$ is not short in the $q$-metric (i.e. $\ell^2_q(\alpha) \geq \epsilon$). 
If the geodesic $\gamma$ crosses the horoball $H = H_\epsilon(\alpha)$ and $t$ is larger than the exit time of $\gamma$ 
from $H$, then 
$$d_\alpha(m_0, m_{t}) \asymp \frac{A}{\epsilon} E(\gamma, H)$$
where $m_t$ is the shortest marking on $\gamma_t$, and the quasi-isometry constants depend only on $X_0$, $\epsilon$ and the topology of $S$.
\end{proposition}
Before proving the proposition, let us recall the definition of extremal length:
$$\textup{Ext}_\sigma(\alpha) := \sup_{\rho} \frac{(\ell_\rho(\alpha))^2}{A(\rho)}$$
where the sup is taken over all metrics $\rho$ in the same conformal
class as $\sigma$.  For any quadratic differential $q$ with area $1$
and any curve $\alpha$,
$$(\ell_q(\alpha))^2 \leqslant \textup{Ext}_\sigma(\alpha) \leqslant
\frac{L_\sigma(\alpha)}{2} e^{L_\sigma(\alpha)/2}$$
where the left-hand side is by definition, while the right-hand side
is due to Maskit \cite{Ma}, and $L_\sigma(\alpha)$ is the length of
$\alpha$ in the hyperbolic metric corresponding to the conformal
structure $\sigma$.

Recall $tw_q^\pm$ denotes the twist parameter in the flat metric 
associated to $q$, as defined in the previous section. Analogously, given a hyperbolic metric 
$\sigma$ on $S$ and a simple closed curve $\alpha$, we can define a twist parameter $tw_\sigma^\pm(\alpha)$
with respect to the hyperbolic metric by taking a curve $\beta$
perpendicular to $\alpha$ with respect to the hyperbolic metric and
letting
$$tw_\sigma^\pm(\alpha) := i(F^\pm, \beta).$$
The following proposition of Rafi relates the two twist parameters:

\begin{proposition}[\cite{Raf07}*{Theorem 4.3}] \label{twistvstwist} %
The two twist parameters are the same up to an additive error
comparable to $1/L_\sigma(\alpha)$. That is,
$$tw^{\pm}_\sigma(\alpha) = tw^\pm_q(\alpha) + O \left( \frac{1}{L_\sigma(\alpha)} \right).$$
\end{proposition}

\begin{proof}[Proof of Proposition \ref{subsurface}]
Let us choose $\epsilon_0$ in such a way that if $\ell^2_q(\alpha) = \epsilon$, then 
$L_\sigma(\alpha) \geqslant \epsilon_0$. 
Let $t_1$ and $t_2$ be the times when the $q_t$-length of $\alpha$ is exactly $\epsilon$, 
and let $t > t_2$.
By the previous choice, $L_{\sigma_{t_i}}(\alpha) \geqslant \epsilon_0$ for $i = 1, 2$, so
by Proposition \ref{disjointinterval} $[0, t_1]$ and $[t_2, t]$ are disjoint from $I_\alpha$, 
hence by Proposition \ref{noprogress}

$$d_\alpha(m_0, m_t) = d_\alpha(m_{t_1}, m_{t_2}) + O(1).$$
On the other hand, the progress in subsurface projection across the horoball is comparable to the progress in the twist parameter 
for the hyperbolic metric, 
$$d_\alpha(m_{t_1}, m_{t_2}) =  
|i_\alpha(\beta_{t_1}, F^+) - i_\alpha(\beta_{t_2}, F^+)| +O (1)= 
|tw_{\sigma_1}^+(\alpha) - tw_{\sigma_2}^+(\alpha)| + O(1)$$
and by Proposition \ref{twistvstwist} it is also comparable to the progress in the twist parameter defined via the flat metric:
$$|tw_{\sigma_1}^+(\alpha) - tw_{\sigma_2}^+(\alpha)| = |tw_{q_1}^+(\alpha) - tw_{q_2}^+(\alpha)| + O\left( \frac{1}{L_{\sigma_1}(\alpha)} \right) + O\left( \frac{1}{L_{\sigma_2}(\alpha)} \right)$$
and since $L_{\sigma_i}(\alpha) \geqslant \epsilon_0$
$$ d_\alpha(m_0, m_t) = |tw_{q_1}^+(\alpha) - tw_{q_2}^+(\alpha)| + O(1).$$
Finally, by Proposition \ref{twist} the twist is comparable to the excursion, thus
$$d_\alpha(m_0, m_t) \asymp \frac{A}{\epsilon}E(\gamma, H).$$
\end{proof}

\subsection{Coarse monotonicity for the word metric} \label{section:coarse}

In \cite{Raf10}, Rafi shows the following non-backtracking or reverse
triangle inequality for subsurface projections along a Teichm\"uller
geodesic. Recall that given a Teichm\"uller geodesic $\gamma_t$ we write $m_t$ for
the shortest marking at $\gamma_t$, and we write $d_Y(m_s, m_t)$ to
mean the distance in the curve complex $\mathcal{C}(Y)$ between the
images of $m_s$ and $m_t$ under subsurface projection to $Y$.

\begin{theorem}[\cite{Raf10}*{Theorem 6.1}] \label{theorem:reverse
  triangle} %
There exists a constant $C$, only depending on the topology of $S$,
such that for every Teichm\"uller geodesic $\gamma$, and every
subsurface $Y$,
\begin{equation} \label{eq:reverse triangle}
d_Y(m_{r}, m_{s}) + d_Y(m_s, m_t) \leqslant d_Y(m_r, m_t) +
C,
\end{equation}
for all constants $r \leqslant s \leqslant t$. 
\end{theorem}

The above theorem along with the Masur-Minsky quasi-distance formula
\eqref{eq:quasi-distance} implies that the distance in the marking complex is coarsely monotonic
along a Teichm\"uller ray.

\begin{proposition}\label{coarse-mono}
There exists constants $C_1 > 0$ and $C_2$ that depend only on $S$
such that along a Teichm\"uller geodesic $\gamma_t$, for $0 < s < t$
the distance in the marking complex satisfies
\[d_M(m_0, m_s) \leqslant C_1 d_M(m_0, m_t) + C_2. \]
\end{proposition} 

\begin{proof}
Let $C$ be the constant in Rafi's reverse triangle inequality,
Theorem \ref{theorem:reverse triangle}.  Assume $0 < s < t$, then
\eqref{eq:reverse triangle} implies
\begin{equation} \label{eq:mono1}  
d_Y(m_0, m_t) \geqslant d_Y(m_0, m_s) - C
\end{equation}
for all subsurfaces $Y \subseteq S$. The Masur-Minsky quasi-distance
formula (Theorem \ref{Quasi-Distance}) holds for all floor constants
sufficiently large, though the quasi-isometry constants depend on
$A$. Choose a floor constant $A > 2 C$, and let $K_1$ and $K_2$ be
the associated quasi-isometry constants.  By the definition of the
floor function, if $\lfloor x \rfloor_A$ is non zero, then $x \geqslant
A$. This implies that $ x - A/2 \geqslant x/2 $, and as the floor
function is monotonic,
\begin{equation} \label{eq:floor} %
\lfloor x - A/2 \rfloor_A \geqslant \lfloor x/2 \rfloor_A.
\end{equation} 
As we have chosen $A > 2 C$, combining \eqref{eq:mono1} and
\eqref{eq:floor} implies
\begin{equation} \label{eq:mono2} %
\lfloor d_Y(m_0, m_t) \rfloor_A \geqslant \lfloor \tfrac{1}{2}
d_Y(m_0, m_s) \rfloor_A,
\end{equation}
again for all subsurfaces $Y \subseteq S$. Now summing
\eqref{eq:mono2} over all subsurfaces $Y \subseteq S$, the
quasi-distance formula implies
\begin{align*}
d_M(m_0, m_t) & \geqslant \frac{1}{K_1} \left( \sum \lfloor
\tfrac{1}{2} d_Y(m_0, m_s)\rfloor_A - K_2 \right). \\
\intertext{By definition of the floor function, $\lfloor \tfrac{1}{2}
  x \rfloor_A = \tfrac{1}{2} \lfloor x \rfloor_{2A}$, so}
d_M(m_0, m_t) & \geqslant \frac{1}{2 K_1} \left( \sum
\lfloor d_Y(m_0, m_s) \rfloor_{2A} - 2 K_2
\right). \\
\intertext{The quasi-distance formula holds for all $A$ sufficiently
  large, so in particular holds for $2A$, though with different
  quasi-isometry constants, which we shall denote $K_3$ and
  $K_4$. This implies that}
d_M(m_0, m_t) & \geqslant \frac{1}{2K_1 K_3} \left(
d_M(m_0, m_s) - K_3 K_4 - K_2 \right)
\end{align*}
whence the result.
\end{proof}

\subsection{Projection to closest Teichm\"uller lattice point} \label{section:projection}

Let $q$ be a quadratic differential, let $q_t$ be the image of $q$
under the Teichm\"uller geodesic flow after time $t$, and let $X_t$ be the image of $q_t$
in $\T$. The orbit of $X_0$ under the mapping class group is called a
\emph{Teichm\"uller lattice}, and let $h_t X_0$ be a choice of closest
lattice point in $\T$ to $X_t$, i.e. such that 
\[ d_\mathcal{T}(h_t X_0, X_t) \leqslant d_{\mathcal{T}}(h X_0, X_t) 
\text{ for all } h \in \textup{Mod}(S). \]
For any given point $X_t$, there are at most finitely many closest
lattice points, however it is possible that the number of closest
lattice points increases as you choose points deeper in the thin part.
Let $m_t$ be a shortest marking on $X_t$, and $d_G$ the word metric on
the mapping class group with respect to some choice of generators.

\begin{lemma} \label{latticepoint} If $X_0$ and $X_t$ both lie in the thick part $\T
  \setminus \mathcal{T}_\epsilon$, then
$$ d_G(1, h_t) \asymp d_M(m_0, m_t)$$
%
where the quasi-isometry constants only depend on $X_0$, the choice of $\epsilon$ and 
the generating set for the mapping class group.
\end{lemma}

\begin{proof}
Let $K_1$ be the diameter of the thick part $\mathcal{T} \setminus \mathcal{T}_{\epsilon}$ in moduli space; 
then, by definition there exists 
a group element $g$ such that in Teichm\"uller space $d_\mathcal{T}(gX_0, X_t) \leqslant K_1$, so 
by definition of $h_t$
%
$$ d_\mathcal{T}(h_t X_0, X_t) \leqslant K_1. $$
Hence by group invariance
$$ d_\mathcal{T}(X_0, h_t^{-1}X_t) \leqslant K_1. $$
In the Teichm\"uller ball of radius $K_1$ only finitely many markings
appear as short markings, hence there exists $K_2$, depending only on
$K_1$, and the surface $S$, such that the distance in the marking complex is bounded:
$$ d_M(m_0, h_t^{-1} m_t) \leqslant K_2. $$
As a consequence, 
$$|d_M(m_0, m_t) - d_M(m_0, h_tm_0)| \leqslant
d_M( h_tm_0, m_t) = d_M(m_0, h_t^{-1} m_t)
\leqslant K_2.$$
Finally, the distance in the word metric $d_G(1, h_t)$ is quasi-isometric to the distance $d_M(m_0, h_t m_0)$
in the marking complex by Proposition \ref{mod-marking}.
\end{proof}

By combining the previous lemma with the coarse monotonicity statement of Proposition \ref{coarse-mono}, we
get that the word length of the closest point projection to the Teichm\"uller lattice is coarsely monotone 
along the thick part of a Teichm\"uller ray:

\begin{proposition} \label{coarse-mono2}
There exists constants $C_1 > 0$ and $C_2$, that depend only on $X_0$ and $\epsilon_0$ and the choice of generators, 
such that along a  Teichm\"uller geodesic $\gamma_t$, for $0 < s < t$ the word metric satisfies
\[d_G(1, h_s) \leqslant C_1 d_G(1, h_t) + C_2 \]
whenever $\gamma_0$, $\gamma_s$ and $\gamma_t$ all lie in the thick part $\T \setminus \T_{\epsilon_0}$.
\end{proposition}

\section{Lebesgue measure sampling} \label{section:lebesgue}

The goal of this section is to study the asymptotic behaviour of
typical Teichm\"uller geodesics with respect to Lebesgue measure,
proving the first part of Theorem \ref{mcg}.  More precisely, we want to keep track
of short curves in the flat metric as the metric changes under the
action of Teichm\"uller flow, and prove an asymptotic result, Theorem
\ref{theorem:asylength}. In Section \ref{section:metric} we recall
results of Masur \cite{MasurLog} and Eskin and Masur
\cite{eskin-masur} which show that the growth rate of the number of
metric cylinders with area bounded below is quadratic. In Section
\ref{section:asymptotic} we consider the function given by the sum of
the squares of the reciprocals of the short curves, and show that the
average value of this function tends to infinity along almost every
Teichm\"uller geodesic with respect to Lebesgue measure. Then in
Section \ref{section:average} we show that this function gives a lower
bound for the average of the sums of the excursions along the
geodesic. Finally in Section \ref{wmetric} we show that the sum of the
excursions is a lower bound for the word metric along the
Teichm\"uller geodesic, and so the word metric along the geodesic has
faster than linear growth, which completes the proof of the Theorem \ref{mcg}
for the Lebesgue measure.

\subsection{Metric cylinders with bounded area} \label{section:metric}

Let $q$ be a quadratic differential of unit area. A \emph{metric cylinder} for $q$ 
is a cylinder in the flat metric associated to $q$ which is the union of freely homotopic closed 
trajectories of $q$. We shall label each metric cylinder by the homotopy class $\alpha$ of the corresponding
closed trajectory. 

Let us now fix some $0 < \delta < 1$, and let $C_q(\delta)$ be the set of metric cylinders
for the $q$-metric with area bounded below by $\delta$.
Moreover, let us denote by $C_q(\delta, \epsilon)$ the set of cylinders whose area is bounded below by $\delta$ 
and whose core curve has length shorter than the square root of $\epsilon$:  
$$C_q(\delta, \epsilon) := \{ \alpha \in C_q(\delta) \ : \ \ell_q^2(\alpha) \leqslant \epsilon \}.$$

\begin{lemma} \label{disj}
Suppose $\epsilon < \delta$. Then any two distinct elements of $C_q(\delta, \epsilon)$ are disjoint on $q$.
As a corollary, the cardinality of $C_q(\delta, \epsilon)$ is bounded above by a constant which depends only 
on the topology of $S$.
\end{lemma}

\begin{proof}
We follow the argument in \cite{MasurLog}*{ Lemma 2.2}. Denote by
$\alpha$ the core curve of some cylinder which belongs to $C_q(\delta,
\epsilon)$. Since the metric cylinder of $\alpha$ has area $A(\alpha)
\geqslant \delta$, any curve $\tau$ which crosses $\alpha$ is such
that $\delta \leqslant \ell_q(\alpha) \ell_q(\tau) \leqslant
\ell_q(\tau) \sqrt{\epsilon}$, hence $\ell_q(\tau) > \sqrt{\epsilon}$,
so $\tau$ cannot belong to $C_q(\delta, \epsilon)$.
\end{proof}

Given the quadratic differential $q$, let us denote as $N_q(\delta, T)$ 
the number of cylinders in the $q$-metric which have area bounded below by $\delta$ and 
length smaller than $T$. As Eskin and Masur showed, 
$N_q(\delta, T)$ grows quadratically as a function of $T$:

\begin{theorem} \label{theorem:QuadAsymp}
There exists $0 < \delta < 1$ and a constant $c_\delta >0$ such that, for almost every quadratic differential $q$ of unit area, we have 
$$\lim_{T \to \infty}\frac{N_q(\delta, T)}{T^2} = c_\delta. $$ 
\end{theorem}

\begin{proof}
Let $0 < \delta < 1$. By the general counting argument of Eskin-Masur \cite{eskin-masur}*{Theorem 2.1} applied to the set of metric 
cylinders with area bounded below by $\delta$, we get the existence of the limit $c_\delta$ almost everywhere. 
On the other hand, by \cite{MasurLog}*{Proposition 2.5}, for every quadratic differential there exists some $\delta > 0$
such that $\liminf_{T \to \infty} \frac{N_q(\delta, T)}{T^2} > 0$, so the constant $c_\delta$ must be positive for some $\delta$. 
\end{proof}

A finer statement, at least in the case of translation surfaces, is due to Vorobets \cite{Vo}.

\subsection{Asymptotic length of short curves} \label{section:asymptotic}

Let us now quantify the idea of keeping track of short curves in the
flat metric. For the rest of the paper, we will fix some $\delta > 0$
for which Theorem \ref{theorem:QuadAsymp} holds, and some $\epsilon <
\delta$.  Let us define the function $L : \mathcal{QM} \to \mathbb{R}$
as
$$L(q):= \sum_{\alpha \in C_q(\delta, \epsilon)} \frac{1}{\ell^2_q(\alpha)}.$$
Note that by Lemma \ref{disj} the number of terms in the sum is always
finite, so the function is well-defined. Let us fix denote by $q_t$
the image of the quadratic differential $q$ under the Teichm\"uller
geodesic flow after time $t$. Our goal is to prove that the ergodic
average of $L$ is infinite:

\begin{theorem} \label{theorem:asylength}
For $\mu_{hol}$-a.e. quadratic differential $q$ of unit area, we have
$$\lim_{T \to \infty} \frac{\int_0^T L(q_t) \ dt}{T} = \infty.$$
\end{theorem}

In the proof of Theorem \ref{theorem:asylength}, we will make use of
the following relations between metric cylinders and the geometry of
Teichm\"uller discs. Let us fix a base point $q_0$ in the space of
quadratic differentials, and call $D_{q_0}$ the Teichm\"uller disc
given by the $SL_2(\mathbb{R})$-orbit of $q_0$. For every metric
cylinder $\alpha$ on $q_0$, there is an angle $\theta_\alpha$ such
that $\alpha$ is vertical in the quadratic differential $e^{i
  \theta_\alpha}q_0$. The angle $\theta_\alpha$ determines a point in
the circle at infinity of $D_{q_0}$.  For each metric cylinder on
$q_0$ with core curve $\alpha$, let us define the set
$$H_\epsilon(\alpha) := \{ q \in D_{q_0} \ : \ \ell_q^2(\alpha)
\leqslant \epsilon \}, $$
of points in the Teichm\"uller disc for which the length of $\alpha$
is less than the square root of $\epsilon$. Recall the metric induced
on $D_{q_0}$ by the Teichm\"uller metric is the hyperbolic metric of
constant curvature $-4$, and $H_\epsilon(\alpha)$ is a horoball for
that metric.

\begin{lemma} \label{horosize}
The Euclidean diameter $s$ of the horoball $H_\epsilon(\alpha)$ is
$$  s = \frac{2 \epsilon}{\epsilon + \ell_{q_0}^2(\alpha)}$$
where $\ell_{q_0}(\alpha)$ is the length of $\alpha$ in the flat metric
associated to the quadratic differential $q_0$.  
\end{lemma}

\begin{proof}
By integrating the hyperbolic metric of curvature $-4$ we have 
$$d(q_0, H_\epsilon(\alpha)) = \int_0^{1-s} \frac{dx}{1-x^2} = \frac{1}{2} \log \frac{2-s}{s}$$
and, since the Teichm\"uller map exponentially shrinks the curve $\alpha$, 
$$e^{-2 d(q_0, H_\epsilon(\alpha))} \ell_{q_0}^2(\alpha) = \epsilon$$
hence the claim.
\end{proof}

We will need the following estimate from elementary Euclidean geometry:

\begin{lemma} \label{angeucl} %
In the unit disc, let $\theta(r, R)$ be the angle at the center of the
disc corresponding to the intersection of the circle of radius $R
\geqslant \frac{1}{2}$ centered at the origin, with a circle of radius
$r \leqslant \frac{1}{2}$ tangent to the boundary, with
$R+2r-1\geqslant 0$. Then there is a constant $K$ such that
\[ \frac{1}{K} \sqrt{(1-R)(R+2r-1)} \leqslant \theta(r, R) \leqslant K
\sqrt{(1-R)(R+2r-1)}. \]
\end{lemma}

\begin{proof} 
By the law of cosines, $r^2 = (1-r)^2 + R^2 - 2R (1-r) \cos
(\theta/2)$. The claim follows by standard algebraic manipulation and
approximation.
\end{proof}

Let $q_{t, \theta}$ denote the quadratic differential given by flowing
the quadratic differential $e^{i \theta} q_0$ for time $t$.

\begin{lemma} \label{angularest}
For almost every quadratic differential $q_0$ there exists a constant $c > 0$,  
such that for each $\epsilon > 0$ there exists a time $t_\epsilon$
such that
$$\sum_{\alpha \in C_{q_0}(\delta)} \leb(\{ \theta \in [0, 2\pi] \ : \ q_{t,
  \theta} \in H_{\epsilon}(\alpha) \})  \geqslant c \epsilon \qquad \forall
\ t \geqslant t_\epsilon, $$
where $\leb$ denotes Lebesgue measure on the circle. 
\end{lemma}

\begin{proof}
Let $\frac{1}{2} < R < 1$, and consider the set of horoballs of the collection
$H_\epsilon(\alpha)$ with $\alpha \in C_{q_0}(\delta)$ and Euclidean diameter $s \geqslant
\frac{3}{2}(1-R)$. By Lemma \ref{horosize}, these horoballs correspond
precisely to metric cylinders with core curve $\alpha$ such that
$$\ell_{q_0}^2(\alpha) \leqslant \frac{3R+1}{3(1-R)} \epsilon.$$ 
By Theorem \ref{theorem:QuadAsymp}, the number of such cylinders is, 
for $R$ large, at least $\frac{c_\delta}{2} \frac{3R+1}{3(1-R)} \epsilon$.
By Lemma \ref{angeucl}, every corresponding
horoball intersects the circle of Euclidean radius $R$ centered at the
origin in an arc of visual angle
$$\theta \geqslant \frac{1}{K \sqrt{2}} (1-R)$$
and by Lemma \ref{disj} every quadratic differential belongs to at
most a universally bounded number $M$ of horoballs, hence the total visual
angle is at least $\frac{c_\delta}{6 KM \sqrt{2}} \epsilon$.
\end{proof}

In order to prove Theorem \ref{theorem:asylength}, let us first define
a discretized version of $L$. Namely, for each $n$ and $\alpha$ we
denote as $H_n(\alpha)$ the horoball
$$H_n(\alpha) := \{ q \in D_{q_0} \ : \ell_q^2(\alpha) \leqslant 2^{-n} \epsilon \}.$$
Now, the function $\Psi : \mathcal{QM} \to \mathbb{R}$ is defined as 
$$\Psi(q) := \sum_{\alpha \in C_q(\delta)} \sum_{n = 1}^\infty 2^n \chi_{H_n(\alpha)}.$$
It is easy to see that $\Psi$ is bounded above by a multiple of $L$: 

\begin{lemma}  \label{comparable}
For each quadratic differential $q$, we have
$$\Psi(q) \leqslant 4\epsilon  L(q).$$
\end{lemma}

\begin{proof}
 
Let $\alpha \in C_q(\delta)$ be a short curve on $q$: then there
exists a positive integer $M$ such that
$$2^{-M} \epsilon \leqslant \ell_q^2(\alpha) \leqslant 2^{-M+1} \epsilon.$$
Now, since $q$ lies in $H_1(\alpha) \cup \dots \cup H_M(\alpha)$, 
$$\sum_{n = 1}^\infty 2^n \chi_{H_n(\alpha)} \leqslant 1 + 2 + \dots + 2^M \leqslant 2 \cdot 2^M \leqslant \frac{4 \epsilon}{\ell^2_q(\alpha)}$$ 
and summing over $\alpha$ yields the claim.
\end{proof}

\begin{proof}[Proof of Theorem \ref{theorem:asylength}]
By Lemma \ref{comparable}, it is enough to prove the statement for $\Psi$.
Let us now truncate the function $\Psi$ by defining, for each $N$, 
$$\Psi_N(q) := \sum_{\alpha \in C_q(\delta)} \sum_{n = 1}^N 2^n \chi_{H_n(\alpha)}.$$
Let us now fix $N$. By Lemma \ref{disj}, $\Psi_N$ is bounded on the moduli space $\mathcal{MQ}$ 
of unit area quadratic differentials, hence $\mu_{hol}$-integrable; by ergodicity of the
geodesic flow, for a generic Teichm\"uller disc for almost all radial
directions $\theta$ the ergodic average of $\Psi_N$ along the flow
tends to its integral:
$$\lim_{T \to \infty} \int_0^T \frac{\Psi_N(q_{t, \theta}) \ dt }{T} =
\int_{\mathcal{MQ}} \Psi_N(q) \ d\mu_{hol} \qquad \textup{for
  a.e. }\theta.$$
Then, if we integrate both sides w.r.t. to the angular measure $d\theta$ and apply the dominated convergence theorem, 
$$\lim_{T \to \infty} \int_{S^1} d\theta  \int_0^T \frac{\Psi_N(q_{t, \theta}) \ dt }{T} = 
\int_{\mathcal{MQ}} \Psi_N(q) \ d\mu_{hol}$$
and by Fubini
$$\lim_{T \to \infty} \frac{  \int_0^T dt \int_{S^1}  \Psi_N(q_{t, \theta}) \ d\theta  }{T} = 
\int_{\mathcal{MQ}} \Psi_N(q) \ d\mu_{hol}.$$
Now, by Lemma \ref{angularest}, for each $t > T_{2^{-N}}$ 
$$\int_{S^1}  \Psi_N(q_{t, \theta}) \ d\theta \geqslant \sum_{n = 1}^N 2^n \cdot c2^{-n} = cN$$
hence
$$\int_{\mathcal{MQ}} \Psi_N(q) \ d\mu_{hol} = \lim_{T \to \infty} \frac{  \int_0^T dt \int_{S^1}  \Psi_N(q_{t, \theta}) \ d\theta  }{T} \geqslant c N.$$
Since the previous estimate works for all $N$, then also
$$\int_{\mathcal{MQ}} L(q) \ d\mu_{hol} = \infty$$
hence the ergodic average tends to infinity almost everywhere: 
$$\lim_{T \to \infty} \int_0^T \frac{L(q_t) \ dt }{T} =
\int_{\mathcal{MQ}} L(q) \ d\mu_{hol} = \infty \qquad \textup{for
  a.e. }q \in \mathcal{MQ}. $$
\end{proof}

\subsection{Average excursion} \label{section:average}

Let us now turn the asymptotic estimate of the previous section into an asymptotic about excursions.
If $q$ is a quadratic differential, let us denote as $\gamma_q$ the corresponding Teichm\"uller geodesic ray.
We now define the concept of \emph{total excursion} traveled by the geodesic $\gamma_q$ inside the horoballs up to time $T$:

\begin{definition}
Given a quadratic differential $q$, the \emph{total excursion} $E(q, T)$ is 
the sum of all excursions in all horoballs crossed by the geodesic ray $\gamma_q$ up to time $T$:
$$E(q, T) := \sum_{\gamma_q([0, T]) \cap H_\epsilon(\alpha) \neq \emptyset} E(\gamma_q, H_\epsilon(\alpha)).$$
\end{definition}

Our goal is to prove that also the average total excursion is infinite. 

\begin{theorem} \label{total_excursion}
For $\mu_{hol}$-almost every quadratic differential $q$ of unit area, we have
$$\lim_{T \to \infty} \frac{E(q, T)}{T} = \infty.$$
\end{theorem}

Theorem \ref{total_excursion} follows from Theorem \ref{theorem:asylength} and the following

\begin{proposition}
Let $q$ be a quadratic differential with geodesic ray $\gamma_q$, and let $T > 0$ be such that 
both $q$ and $\gamma_q(T)$ lie outside all horoballs of the type $H_\epsilon(\alpha)$. Then
$$\int_0^T L(q_t) \ dt \leqslant \frac{C}{\epsilon} E(q, T)$$
for some universal constant $C$.
\end{proposition}

\begin{proof}
Let $\alpha \in C_q(\delta)$ be a curve which has become short before
time $T$, i.e. such that $\gamma_q([0, T]) \cap H_\epsilon(\alpha)$ is
non-empty. Let $T_1$ be the time the geodesic enters
$H_\epsilon(\alpha)$, and $T_2$ the time the geodesic exits. Moreover,
let $N$ be the maximum integer $k$ such that the geodesic enters
$H_k(\alpha)$. Note that there is a universal constant $C_1$ such that for
each $n \geqslant 1$ and each $\alpha$
$$ \leb ( \{ t \in [0, T] \ : \ q_t \in H_{n+1}(\alpha) \setminus H_n(\alpha) \} ) \leqslant C_1.$$
Then
$$\int_{T_1}^{T_2} \frac{1}{\ell^2_{q_t}(\alpha)} \ dt \leqslant \sum_{n = 1}^N  \frac{2^n}{\epsilon} 
\leb(\{ t \in [0, T] \ : \ q_t \in H_{n+1}(\alpha) \setminus H_n(\alpha) \} ) \leqslant \frac{C_1 \cdot 2^{N+1}}{\epsilon}.$$
In order to compare the right hand side with the excursion, let us denote by $\tilde{\epsilon}$ the smallest value of
 $\ell_{q_t}^2(\alpha)$ along the geodesic ray $\gamma_q$. By the definition of $N$, we have 
$\tilde{\epsilon} \asymp 2^{-N} \epsilon$. Now, by the definition of excursion and Lemma \ref{phimax},
$$E(\gamma_q, H_\epsilon(\alpha)) = \frac{\phi_{max}}{\phi_0} \asymp \frac{\sin \phi_{max}}{\sin \phi_0} = 
\frac{\epsilon}{\tilde{\epsilon}} \asymp 2^N $$
(where all the approximate equalities hold up to multiplicative constants), hence the claim follows.
\end{proof}

\begin{remark}
A precise analysis of how $E(q,T)$ grows along Leb-typical geodesics is carried out in \cite{Ga2}. It culminates in a strong law analogous to the one established by Diamond and Vaaler for continued fractions \cite{Dia-Vaa}.
\end{remark}

\subsection{The word metric} \label{wmetric}

Let us complete the proof of Theorem \ref{mcg} for the Lebesgue measure by proving that
the word metric is bounded below by the total excursion. Let us pick
$\epsilon_0$ to define the thick part as in section
\ref{section:projections}, and let us choose $\delta$ so that Theorem
\ref{theorem:QuadAsymp} holds.  Finally, we choose $\epsilon$ so that
if $X$ belongs to the thick part $\T \setminus \T_{\epsilon_0}$ and
$\alpha$ is the core curve of a metric cylinder of $q$-area larger
than $\delta$ on $X$, then $\ell_q^2(\alpha) \geqslant \epsilon$.

Let $X_0$ lie in the thick part $\mathcal{T} \setminus
\mathcal{T}_{\epsilon_0}$, and let $\gamma$ be a Teichm\"uller
geodesic with $\gamma(0) = X_0$. Recall for each time $t$, $h_t$ is a
closest point projection of $\gamma(t)$ to the Teichm\"uller lattice.

\begin{proposition} \label{thickexcursion}
If $\gamma(T)$ lies in the thick part $\mathcal{T} \setminus \mathcal{T}_{\epsilon_0}$, then %
$$d_G(1, h_T) \geqslant C_1 E(\gamma, T) - C_2$$
where the constants depend only on $X_0$, the choice of $\epsilon_0$ and 
the choice of generating set for $Mod(S)$. 
\end{proposition}

\begin{proof}
Since $\gamma(0)$ and $\gamma(T)$ lie in the thick part, by Lemma \ref{latticepoint}
$$d_G(1, h_T) \asymp d_M(m_0, m_T).$$
By the Masur-Minsky quasi-distance formula (Theorem \ref{Quasi-Distance}), for any $B$ large enough 
$$ d_M(m_0, m_T) \asymp \sum_{Y \subseteq S} \lfloor d_Y(m_0,m_T) \rfloor_B
 \geqslant \sum_{\gamma([0, T]) \cap H_\epsilon(\alpha) \neq \emptyset}  \lfloor d_\alpha(m_0, m_T) \rfloor_B$$
where on the right-hand side we only consider projections to annuli of area bounded below, and whose core curve 
becomes short before time $T$. Now by Proposition \ref{subsurface}, for some constants $K_1$ and $K_2$, 
$$ d_\alpha(m_0, m_T) \geqslant K_1 E(\gamma, H_\epsilon(\alpha)) - K_2 $$
so if $B  \geqslant K_2$
$$ \lfloor d_\alpha(m_0, m_T) \rfloor_B \geqslant \lfloor K_1 E(\gamma, H_\epsilon(\alpha)) -K_2 \rfloor_B \geqslant 
\frac{K_1}{2} \lfloor E(\gamma, H_\epsilon(\alpha)) \rfloor_\frac{2B}{K_1} $$ 
and we can choose $\tilde{\epsilon}$ a bit smaller than $\epsilon$ so that $\lfloor E(\gamma, H_\epsilon(\alpha)\rfloor_{\frac{2B}{K_1}} \geqslant E(\gamma, H_{\tilde{\epsilon}}(\alpha))$,
hence
$$\sum_{\gamma([0, T]) \cap H_\epsilon(\alpha) \neq \emptyset} \lfloor d_\alpha(m_0, m_T) \rfloor_B \geqslant 
\sum_{\gamma([0, T]) \cap H_{\tilde{\epsilon}}(\alpha) \neq \emptyset} E(\gamma, H_{\tilde{\epsilon}}(\alpha)) = E(q, T).$$ 
\end{proof}

\begin{proof}[Proof of Theorem \ref{mcg} (Lebesgue measure)]
By Theorem \ref{unparameterized}, the relative metric is a lower bound for Teichm\"uller distance, i.e. there exist constants
$K, c$, depending only on the topology of $S$, such that 
\[
d_{rel}(1, h_T) \leqslant K T + c.
\]
The first part of Theorem \ref{mcg} then follows from Theorem \ref{total_excursion} and Proposition \ref{thickexcursion}.
\end{proof}

\section{Hitting measure sampling} \label{section:hitting}

In Section \ref{section:random} we review some background material
from the theory of random walks, and recall some previous results
which show that the ratio between the word metric and the relative
metric along the locations $w_n X_0$ of a sample path of the random
walk remains bounded for almost all sample paths. This means that if a
location $w_n X_0$ of the sample path is close to the geodesic, then this
ratio is also bounded for points on the geodesic close to $w_n X_0$.
However, the results of the previous section apply to all points along
the geodesic which lie in the thick part of Teichm\"uller space, so we
need to extend the bounds to these other points.  In Section
\ref{section:distance} we use some results of \cite{T12} to show that
the distance between the locations of the sample path and the
corresponding geodesic grows sublinearly, and then in Section
\ref{section:intermediate} we use the coarse monotonicity of word
length along the geodesic to show that this also bounds the ratio
between word length and relative length for all points along the
geodesic which lie in the thick part.

\subsection{Random walks} \label{section:random}

Let $\mu$ be a measure on the mapping class group
$G = \textup{Mod}(S)$. We say that $\mu$ has \emph{finite first moment} with respect 
to the word metric on $G$ if 
$$\int_{G} d_G(1, g) \ d\mu(g) < \infty$$
where $d_G$ is a word metric on $G$ with respect to a choice of finite set of generators (note that the finiteness does not 
depend on this choice). 
The \emph{step space} is the infinite product
$G^\mathbb{N}$ with the product measure $\mathbb{P} := \mu^\mathbb{N}$. Let $w_n =
g_1 g_2 \ldots g_n$ be the location of the random walk after $n$
steps. The \emph{path space} is $G^\mathbb{N}$, with the pushforward
of the product measure under the map
\[ (g_1, g_2, g_3, \dots) \mapsto ( w_1, w_2, w_3, \ldots). \]
It will also be convenient to consider \emph{bi-infinite} sample
paths. In this case the step space is the set $G^\mathbb{Z}$ of bi-infinite 
sequences of group elements with the product measure. The location of the random walk is given by $w_0 = 1$, and
$w_n = g_1 g_2 \ldots g_n$ if $n$, if $n$ is positive, and $w_n =
g_{0}^{-1} g_{-1}^{-1} \ldots g_{n-1}^{-1}$, if $n$ is negative.  The
path space is $G^{\mathbb{Z}}$, as a set, but with measure coming from
the pushforward of $\mathbb{P}$ under the map
\[ ( \dots, g_{-1}, g_0 , g_1, g_2, \dots) \mapsto ( \ldots w_{-1},
w_0, w_1, w_2, \ldots). \]

Let us fix a base point $X_0 \in \mathcal{T}$, and consider the image
of the sample paths $w_n X_0$ in $\mathcal{T}$. 
Kaimanovich and Masur
showed that almost every sample path converges to a uniquely ergodic
foliation in the space $\mathcal{PMF}$ of projective measured foliations,
Thurston's boundary for Teichm\"uller space.
Recall that the \emph{harmonic measure} $\nu$ on $\mathcal{PMF}$ is defined as the hitting 
measure of the random walk, i.e. for any measurable subset $A \subseteq \mathcal{PMF}$, 
$$\nu(A) := \mathbb{P}( w_n \ : \lim_{n \to \infty} w_n X_0 \in A  ).$$

\begin{theorem}[Kaimanovich and Masur \cite{KM}] 
Let $\mu$ be a probability distribution on the mapping class group
whose support generates a non-elementary subgroup. Then almost every
sample path $(w_n)_{n \in \mathbb{N}}$ converges to a uniquely ergodic
foliation in $\mathcal{PMF}$, and the resulting hitting measure $\nu$
is the unique non-atomic $\mu$-stationary measure on $\mathcal{PMF}$.
\end{theorem}

The mapping class group is finitely generated, and let $d_G$ be the
word metric on the mapping class group with respect to some choice of
generating set. Since the mapping class group is non-amenable, the
random walk makes linear progress in the word metric $d_G$, or indeed
in any metric quasi-isometric to the word metric.

\begin{theorem}[Kesten \cite{kesten}, Day \cite{day}] \label{linword}
Let $\mu$ be a probability distribution on a group, whose support
generates a non-amenable subgroup. Then there exists a constant $c_1 >
0$ such that for almost all sample paths
\begin{equation} \label{linearword}
\lim_{n \to \infty} \frac{d_{G}(1, w_n)}{n} = c_1.
\end{equation}
\end{theorem}

Even though the relative metric is smaller than the word metric, 
more recent results prove that the growth rate is still linear in the
number of steps.

\begin{theorem}[Maher \cite{Mah2}, Maher-Tiozzo \cite{MT}] \label{linrelative}
Let $\mu$ be a probability distribution on the mapping class group which 
has finite first moment in the word metric, and such that the semigroup 
generated by its support is a non-elementary subgroup. 
Then there is a constant $c_2 > 0$ such that for almost all sample paths
$$\lim_{n \to \infty} \frac{d_{rel}(1, w_n)}{n} = c_2.$$
\end{theorem}

Note that in \cite{Mah2}, the result is proven under the additional condition that 
the support of $\mu$ is bounded in the relative metric, while
such condition is not needed in 
\cite{MT}. 
 
From these results it follows that the quotient between the word
metric and the relative metric converges to $c_1 / c_2$ for almost
every sample path, i.e.
\[ \lim_{n \to \infty} \frac{ d_G(1, w_n) }{ d_{rel}(1, w_n) } =
\frac{c_1}{c_2} \]
for almost all sample paths.  The limit above is a limit taken along
the locations $(w_n)_{n \in \mathbb{N}}$ of the random walk. In order
to compare this to the previous statistic we need to relate locations
of the random walk to points on a Teichm\"uller geodesic.

By the work of Kaimanovich and Masur \cite{KM}, for almost every
bi-infinite sample path $w \in G^\mathbb{Z}$, there are well-defined
maps
$$ F^{\pm} : G^\mathbb{Z} \to \mathcal{PMF}$$
given by
\begin{align*}
F^+(w) & :=  \lim_{n \to \infty} w_n X_0 \\
\intertext{and}
F^-(w) & :=  \lim_{n \to \infty} w_{-n} X_0. \\
\end{align*}

Furthermore, the two foliations $F^+(w)$ and $F^-(w)$ are almost
surely uniquely ergodic and distinct, so there is a unique oriented
Teichm\"uller geodesic $\gamma_w$ whose forward limit point in
$\mathcal{PMF}$ is $F^+(w)$ and whose backward limit point is
$F^-(w)$. There is also a unique geodesic ray $\rho_w$ starting at the
basepoint $X_0$ whose forward limit point is $F^+$.  We shall always
parameterize $\rho_w$ as unit speed geodesic with $\rho_w(0) = X_0$.
As $F^+$ is uniquely ergodic, the distance between $\gamma_w$ and
$\rho_w$ tends to zero, by Masur \cite{MasurUniq}, and we shall
parameterize $\gamma_w$ such that $d_\T(\rho_w(t), \gamma_w(t)) \to
0$.

For each bi-infinite sample path we can define the function
$$D : G^\mathbb{Z} \rightarrow \mathbb{R}$$
given by
$$D(w) := d_\mathcal{T}(X_0, \gamma_w)$$
which represents the Teichm\"uller distance between the base point
$X_0$ and the geodesic $\gamma_w$. This is well-defined and
measurable, by Lemma 1.4.4 of \cite{KM}. In particular, this implies
that for any $\epsilon > 0$ there is a constant $M$ such that the
probability that $D(w) \leqslant M$ is at least $1 - \epsilon$.

The shift map $\sigma$ maps the step space to itself by incrementing
the index of each step by one, i.e. 
\[\sigma \colon (g_n)_{n \in \mathbb{Z}} \mapsto (g_{n+1})_{n \in
  \mathbb{Z}}. \]
This is a measure preserving ergodic transformation on the step space,
and the induced action of $\sigma$ on the path space is given by
\[ \sigma \colon (w_n)_{n \in \mathbb{Z}} \mapsto (w_1^{-1}
w_{n+1})_{n \in \mathbb{Z}}. \]

\subsection{Distance between geodesic and sample path} \label{section:distance}

The geodesic $\gamma_w$ is determined by its endpoints $F^+(w)$ and
$F^-(w)$, and the distribution of these pairs is given by harmonic
measure $\nu$ and reflected harmonic measure $\check \nu$
respectively.

The distance from a location $w_n$ to the corresponding geodesic
$\gamma_w$ is given by 
\begin{align*}
d_\mathcal{T}(w_n X_0, \gamma_w) & = d_\mathcal{T}(X_0, w_n^{-1}
\gamma_w) 
\intertext{since the mapping class group
acts on $\mathcal{T}$ by isometries, and by the definition of the shift map,}
d_\mathcal{T}(w_n X_0, \gamma_w) & = d_\mathcal{T}(X_0,
\gamma_{\sigma^n w}). 
\end{align*}

As already noted in \cite{KM}, if $\epsilon$ is sufficiently small,
almost every geodesic with respect to harmonic measure returns to the
$\epsilon$-thick part $\T \setminus \T_\e$ infinitely often.

Our goal is to show that every step of the random walk lies within
sublinear distance in the word metric from some point in the thick
part of the limit geodesic. 

In \cite{T12}, sublinear tracking is proven in the 
Teichm\"uller metric: we will adapt the argument to the word metric. 
The fundamental argument for sublinear tracking in \cite{T12} is the 
following lemma.

\begin{lemma}[Tiozzo \cite{T12}] \label{L:st}
Let $T : \Omega \to \Omega$ a measure-preserving, ergodic transformation
of the probability measure space $(\Omega, \lambda)$, and let $f : \Omega \to \mathbb{R}^{\ge 0}$
any measurable, non-negative function. If the function 
$$g(\omega) := f(T\omega) - f(\omega)$$
belongs to $L^1(\Omega, \lambda)$, then for $\lambda$-almost every $\omega \in \Omega$ one has
$$\lim_{n \to \infty} \frac{f(T^n \omega)}{n} = 0.$$
\end{lemma}

We now explain how to apply the lemma above in the current setting.
Given a point $X \in \mathcal{T}$, let us denote as $\textup{proj}(X)$
the set of lattice points at minimal distance from $X$:
$$\textup{proj}(X) := \{ h \in G :
d_\mathcal{T}( hX_0, X) \textup{ is minimal} \}.$$
Such a projection may possibly vary wildly if $X$ lies in the thin
part, but it is controlled in the thick part: namely, given $\epsilon
> 0$ there is a constant $K(\epsilon)$ such that
$$d_\mathcal{T}(X, h X_0) \leqslant K(\epsilon), \qquad \forall X \notin
\mathcal{T}_\epsilon \quad \forall h \in \textup{proj}(X). $$
We now associate to almost every sample path $w$ a subset $P(w)$ of the
mapping class group, which we now describe.  Almost
every bi-infinite sample path $w \in G^\mathbb{Z}$ determines two
uniquely ergodic foliations, $F^\pm(w)$. Let $\gamma_w$ be the
bi-infinite Teichm\"uller geodesic joining them. Now, let us define
$P(w)$ as the set of mapping class group elements $h \in G$ such that
$hX_0$ is the closest projection from some point $X$ in $\gamma_w
\setminus \T_\e$, i.e.
$$ P(w) := \bigcup_{X \in \gamma_w \setminus
  \mathcal{T}_\epsilon} \textup{proj}(X). $$
This is illustrated in Figure \ref{pic:P(w)}.

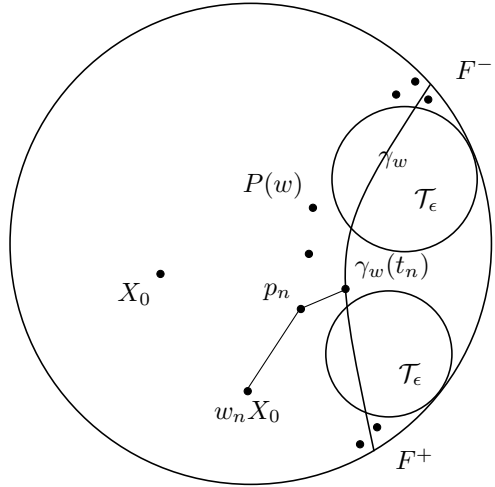
\begin{figure}[H] 
\centering
\begin{tikzpicture}[y=-1cm, scale=0.4]

\draw[semithick,black] (17.59111,13.66) circle (2.09556cm) node [below right] {$\mathcal{T}_\epsilon$};
\draw[semithick,black] (18.11333,7.84667) circle (2.41333cm) node [below right] {$\mathcal{T}_\epsilon$};

\filldraw[black] (18.9,5.2) circle (0.12cm);
\filldraw[black] (17.83333,5.03333) circle (0.12cm);
\filldraw[black] (18.46667,4.6) circle (0.12cm);
\filldraw[black] (17.2,16.1) circle (0.12cm);
\filldraw[black] (16.63333,16.66667) circle (0.12cm);
\filldraw[black] (10,11) circle (0.12cm) node [below left] {$X_0$};
\filldraw[black] (12.9,14.9) circle (0.12cm) node [below] {$w_n X_0$};
\filldraw[black] (14.66667,12.16667) circle (0.12cm) node [above left]{$p_n$};
\filldraw[black] (15.06667,8.8) circle (0.12cm) node [above left] {$P(w)$};
\filldraw[black] (14.93333,10.33333) circle (0.12cm);
\filldraw[black] (16.15,11.5) circle (0.12cm) node [above right] {$\gamma_w(t_n)$};

\draw[black] (16.2,11.5) -- (14.63333,12.16667) -- (12.86667,14.86667);

\draw[semithick,black] (18.96667,4.7) .. controls (15.6,9.86667) .. (17.1,16.9);

\path (19.5,4.5) node[text=black,anchor=base west] {$F^-$};
\path (17.5,17.5) node[text=black,anchor=base west] {$F^+$};
\path (16.93333,7.3) node[text=black,anchor=base west] {$\gamma_w$};

\draw[semithick,black] (13,10) circle (8cm);

\end{tikzpicture}%
\caption{Sample path locations and basepoint orbits close to the geodesic.}
\label{pic:P(w)}
\end{figure}

The key result is the following:

\begin{proposition} \label{sublinear} %
Fix $\epsilon > 0$, sufficiently small. Then for almost every sample
path $(w_n)_{n \in \mathbb{N}}$, with corresponding Teichm\"uller ray
$\rho_w$, there exists a sequence of times $t_n \to \infty$ with
$\rho_w({t_n}) \in \rho_{w} \setminus \mathcal{T}_\epsilon $, such that
\[ \lim_{n \to \infty} \frac{d_G(w_n, h_n)}{n} = 0 \]
for any $h_n \in \textup{proj}(\rho_w({t_n}))$.
\end{proposition}

\begin{proof}
Let us fix $\epsilon > 0$ sufficiently small. Recall that $P(w)$ is
the collection of group elements corresponding to closest lattice
points to points on the geodesic $\gamma_w$ which lie in the thick
part of Teichm\"uller space. Note that, since the mapping class group
acts by isometries with respect to both the Teichm\"uller and word
metrics, then $P$ is equivariant, in the sense that
$$P(\sigma^n w) = w_n^{-1}P(w).$$
Let us now define the function $\varphi : G^\mathbb{Z} \rightarrow
\mathbb{R}$ on the space of bi-infinite sample paths as
$$\varphi(w) := d_G(1, P(w))$$
i.e. the minimal word-metric distance between the base point $X_0$ and
the set of closest projections from the thick part of the geodesic
$\gamma_w$.  
The shift map $\sigma : G^\mathbb{Z} \to G^\mathbb{Z}$ acts on
the space of sequences, ergodically with respect to the product
measure $\mu^\mathbb{Z}$.
By the equivariance of $P$, we have for each $n$ the equality 
\begin{equation}\label{eq:shift}
\varphi(\sigma^n w) = d_G(w_n, P(w)).
\end{equation}

We shall now apply Lemma \ref{L:st}, setting $(\Omega, \lambda) = (G^\mathbb{Z}, \mu^\mathbb{Z})$, 
$T = \sigma$, and $f = \varphi$. The only condition to be checked is the $L^1$-condition 
on the function $g(\omega) = f(T\omega)-f(\omega)$, which in this case becomes
$$g(\omega) = \varphi(\sigma w) - \varphi(w) = d_G(1, P(\sigma w)) - d_G(1, P(w)).$$
Now, using \eqref{eq:shift} we have 
$$|d_G(1, P(\sigma w)) - d_G(1, P(w))| = 
|d_G(w_1, P(w)) - d_G(1, P(w))| \le d_G(1, w_1)$$
which has finite integral precisely by the finite first moment assumption.
Thus, it follows from Lemma \ref{L:st} that for almost all bi-infinite paths $w$ one gets
$$\lim_{n \to \infty} \frac{d_G(w_n, P(w))}{n} = 0.$$
By definition of $P(w)$, there exists a sequence of times $t_n$, such that $\gamma_w(t_n)$ lies in
$\gamma_w \setminus \mathcal{T}_\epsilon$, the $\epsilon$-thick part
of the geodesic $\gamma_w$, and group elements $p_n \in G$ such that
$p_n \in \textup{proj}(\gamma_w(t_n))$, and furthermore
\begin{equation} \label{subpn}
\lim_{n \to \infty} \frac{d_G(w_n, p_n)}{n} = 0. 
\end{equation}
Now let $F^+$ be the terminal foliation of the geodesic $\gamma_w$,
and denote as $\rho_w$ the geodesic ray through $X_0$ with terminal
foliation $F^+$. We have obtained a sequence of points lying in the
intersection of the geodesic $\gamma_w$ with the thick part $\T
\setminus \T_\e$, and we now show how to obtain a sequence of points
lying in the intersection of the geodesic $\rho_w$ with the thick part
$\T \setminus \T_\e$.

Recall that since $\gamma_w$ and $\rho_w$ have the same terminal
foliation $F^+$, and $F^+$ is almost surely uniquely ergodic, then the
distance between the positive ray $\rho_w$ and the geodesic $\gamma_w$
tends to zero, and we have chosen parameterizations such that
$d_\T(\gamma_w(t), \rho_w(t)) \to 0$. In particular, after discarding
finitely many initial values, we may assume
$$ d_\mathcal{T}(\gamma_w(t_n), \rho_w(t_n)) \leqslant \frac{\log 2}{2}, $$
for all $n$. Now for each $n$ sufficiently large consider the sequence
take $\rho_w(t_n)$. Then:
\begin{enumerate}
\item By Wolpert's lemma, $\rho_w({t_n})$ lies in the $\frac{\epsilon}{2}$-thick part;
\item if $h_n \in \textup{proj}(\rho_w({t_n}))$, then
$d_\mathcal{T}(\rho_w({t_n}), h_n X_0) \leqslant K(\epsilon/2)$ so
\begin{align*}
d_\mathcal{T}(h_n X_0, p_n X_0) & \leqslant d_\mathcal{T}(h_n X_0,
\rho_w({t_n})) + d_\mathcal{T}(\rho_w({t_n}), \gamma_w(t_n)) +
d_\mathcal{T}(\gamma_w(t_n), p_n X_0) \\
& \leqslant 1 + 2 K(\epsilon/2),
\end{align*}
hence $d_G(h_n, p_n) \leqslant K'$, so by equation \eqref{subpn} we have also 
$$\lim_n \frac{d_G(w_n, h_n)}{n} \to 0.$$
\end{enumerate}
This completes the proof of Proposition \ref{sublinear}.
\end{proof}

\subsection{Intermediate times} \label{section:intermediate}

So far, we have shown that every step of the sample path is close
enough to some point on the thick part of the geodesic, hence the
closest projection to the lattice will behave like the sample path.
However, we still need to deal with the case in which there are points
in the thick part of the Teichm\"uller geodesic which are not close to
the sample path.

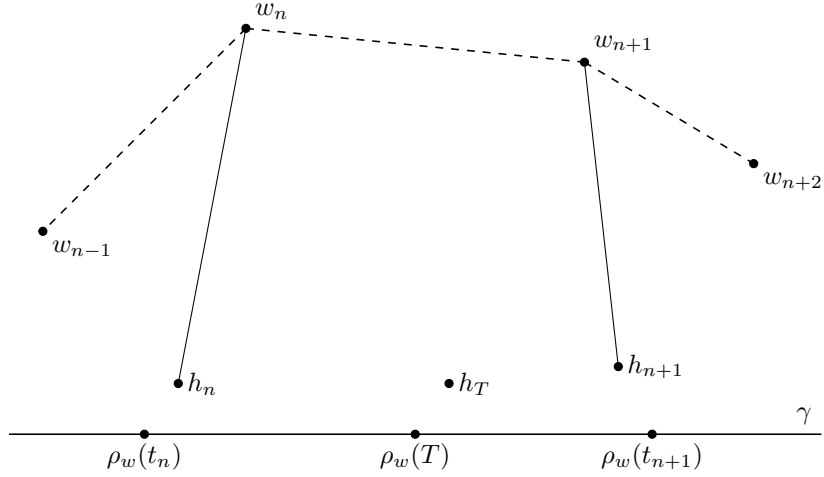
\begin{figure}[H]
\centering
\begin{tikzpicture}[y=-1cm, scale=0.45]


\filldraw[black] (5,14) circle (0.12cm) node [below right] {$w_{n-1}$};
\filldraw[black] (11,8) circle (0.12cm) node [above right] {$w_n$};
\filldraw[black] (21,9) circle (0.12cm) node [above right] {$w_{n+1}$};
\filldraw[black] (26,12) circle (0.12cm) node [below right] {$w_{n+2}$};
\filldraw[black] (8,20) circle (0.12cm) node [below] {$\rho_w(t_n)$};
\filldraw[black] (23,20) circle (0.12cm) node [below] {$\rho_w(t_{n+1})$};
\filldraw[black] (16,20) circle (0.12cm) node [below] {$\rho_w(T)$};
\filldraw[black] (9,18.5) circle (0.12cm) node [right] {$h_n$};
\filldraw[black] (22,18) circle (0.12cm) node [right] {$h_{n+1}$};
\filldraw[black] (17,18.5) circle (0.12cm) node [right] {$h_T$};
\draw[semithick,black] (4,20) -- (28,20) node [above left] {$\gamma$};
\draw[semithick,dashed,black] (11,8) -- (21,9);
\draw[semithick,dashed,black] (11,8) -- (5,14);
\draw[semithick,dashed,black] (21,9) -- (26,12);
\draw[black] (11,8) -- (9,18.5);
\draw[black] (21,9) -- (22,18);

\end{tikzpicture}%
\caption{Intermediate times.}
\end{figure}

\begin{proof}[Proof of Theorem \ref{mcg} (harmonic measure)]
Given a sample path $w$, let $\rho_w$ be the geodesic ray joining the base point $X_0$ 
to the limit foliation $F^+(w)$, and let $t_n$ be the sequence of times given by Proposition \ref{sublinear}.
Let now $T > 0$ be a time for which the geodesic $\rho_w(T)$ lies in the thick part, 
and let $h_T X_0$ be a projection of $\rho_w(T)$ to the Teichm\"uller lattice. Since $t_n \to \infty$, 
there exists an index $n = n(T)$ such that $t_n \leqslant T \leqslant t_{n+1}$.
By Proposition \ref{coarse-mono2}, there exist constants $C_1 > 0$, $C_2$ such that
$$ d_G(h_n, h_T) \leqslant C_1 d_G(h_n, h_{n+1}) + C_2.$$ 
Moreover, by Proposition \ref{sublinear} and triangle inequality,
$$\lim_{n \to \infty} \frac{d_G(h_n, h_{n+1})}{n} \leqslant 
\lim_{n \to \infty} \frac{d_G(h_n, w_n) + d_G(w_n, w_{n+1}) +
  d_G(w_{n+1}, h_{n+1})}{n} = 0$$
(where we used the finite first moment condition to ensure $d_G(w_n, w_{n+1})/n \to 0$). Thus, we also have
$$ \lim_{n \to \infty} \frac{d_G(h_n, h_T)}{n} = 0 $$ 
and again by Proposition \ref{sublinear}
$$ \lim_{n \to \infty} \frac{d_G(w_n, h_T)}{n} \leqslant \lim_{n \to \infty} \frac{d_G(w_n, h_n) + d_G(h_n, h_T)}{n} = 0 .$$ 
Similarly, since the relative metric is bounded above by the word metric, 
$$ \lim_{n \to \infty} \frac{d_{rel}(w_n, h_T)}{n} = 0 .$$ 
Finally, by computing the ratio between the word and relative metric,
$$ \lim_{\stackrel{T \to \infty}{\rho_w(T) \notin \T_\epsilon}} \frac{d_G(1, h_T)}{d_{rel}(1, h_T)} = 
\lim_{\stackrel{T \to \infty}{\rho_w(T) \notin \T_\epsilon}} \frac{\frac{d_G(1, h_T)}{n(T)}}{\frac{d_{rel}(1, h_T)}{n(T)}} = 
\lim_{n \to
  \infty} \frac{\frac{d_G(1, w_n)}{n}}{\frac{d_{rel}(1, w_n)}{n}} =
\frac{c_1}{c_2} > 0 .$$
This completes the proof of Theorem \ref{mcg}.
\end{proof}

\section{Fuchsian groups} \label{section:fuchsian}

Let $G$ be a Fuchsian group, i.e. a discrete subgroup of $SL(2,
\mathbb{R})$, with the further property that the quotient $X = G
\backslash \mathbb{H}^2$ is a finite area non-compact orbifold. Such a
subgroup is also known as a nonuniform lattice in $SL(2,
\mathbb{R})$. Given $\epsilon > 0$, the 
\emph{thin part} of $X$ is the set of points with injectivity radius
smaller than $\epsilon$. The complement of the thin part is a compact set 
called \emph{thick part} of $X$, and denote by $N$. 
If $\epsilon$ is sufficiently small, then the thin part is the union of
disjoint neighbourhoods 
$c_1, \cdots, c_p$ of the cusps of $X$. The universal cover of $X$ is the hyperbolic plane 
$\mathbb{H}^2$, and the lift of the union $c_1 \cup \dots \cup c_p$ of the cusp neighbourhoods in 
the universal cover is the union of countably many disjoint horoballs, which we shall denote by $\mathcal{H}$.

The group $G$ is finitely generated, and a finite choice of generators $\mathcal{A}$ for $G$ defines a 
proper {\em word metric} on $G$. Different choices of generators produce quasi-isometric metrics. 
For each cusp neighbourhood $c_i$ in $X$, let us choose a lift $\tilde{c}_i$ in the universal cover, 
and denote by $G_i$ the stabiliser of $\tilde{c}_i$. The group $G_i$ is infinite cyclic and is a maximal parabolic subgroup;
 let $g_i$ be a generator of $G_i$.
We may also define a {\em relative metric} on $G$ by taking the word metric with respect to the larger (infinite)
generating set 
$$\mathcal{A}' := \mathcal{A} \cup G_1 \cup \dots \cup G_p;$$
that is, along with the generators of $G$, the set $\mathcal{A}'$ includes all powers of all the parabolic generators $g_i$. 
The metric space $(G, d_{rel})$ is not proper, but it is Gromov hyperbolic. In fact,
as proven by Farb, $G$ is strongly hyperbolic relative to the parabolic subgroups $G_i$ \cite{Far}*{Theorem 4.11}.

Recall that a subgroup $G$ of $SL(2, \mathbb{R})$ is called \emph{non-elementary}
if it contains a pair of hyperbolic isometries with different fixed points. 
Let $\mu$ be a  measure on $G$, such that the
support of $\mu$ generates a non-elementary subgroup of $SL(2, \mathbb{R})$ as a semigroup, 
and consider the random walk generated by $\mu$. That is, the space $G^\mathbb{N}$
of sequences $(g_1, g_2, \dots)$ is endowed with the product measure $\mu^\mathbb{N}$, 
and we define the random walk as the process $\{w_n \}_{n \geq 0}$ with 
$w_0 = id$ and 
$$w_{n+1} =  w_n g_{n+1}.$$
Given a basepoint $x_0 \in \mathbb{H}^2$, one can consider the orbit map $G
\to \mathbb{H}^2$ which sends $g \mapsto g(x_0)$, so each sample path
in $G$ projects to a sample path in $\mathbb{H}^2$. Furstenberg
\cite{Furstenberg} showed that for almost all sequences 
the random walk converges to some point $p_{\infty} \in S^1 = \partial
\mathbb{H}^2$. The harmonic measure $\nu$ on the boundary records the
probability that the random walk hits a particular part of $\partial
\mathbb{H}^2$, i.e.
\[
\nu(A) = \textup{Prob}\left(\lim_{n \to \infty} w_n (x_0) \in
A\right).
\]

The unit tangent bundle $T^1 \mathbb{H}^2$ carries a natural 
$SL(2, \mathbb{R})$-invariant measure,  which in the upper half-plane model is given by $d \ell=
\frac{dx\ dy\ d \theta}{y^2}$. 
This measure descends to a measure on the unit tangent bundle to $X = G\backslash \mathbb{H}^2$ 
which is invariant for the geodesic flow, and is called \emph{Liouville measure}. 
Moreover, it is a classical result due to Hopf \cite{Hopf} 
that this flow is ergodic, and indeed mixing. The conditional
measure on the unit circle in the tangent space at any point is the
pullback via the visual map of the standard Lebesgue measure on
$\partial \mathbb{H}^2 = S^1$.

By studying the collection $\mathcal{H}$ of horoballs, Sullivan \cite{Sul} showed
that a generic geodesic ray with respect to Lebesgue measure is
recurrent to the thick part of $X$, and ventures into the cusps
infinitely often with maximum depth in the cusps of about $\log t$,
where $t$ is the time along the geodesic ray. Sullivan's theorem is a
precursor of Masur's approach in Teichm\"uller space \cite{MasurLog}.

Given a horoball $H$ and a geodesic $\gamma$ that
enters and leaves $H$, we define the excursion $E(\gamma, H)$ to be
the distance in the path metric on $\partial H$ between the entry and
exit points. 
Sullivan's theorem implies that a lift in $\mathbb{H}^2$ of a
Lebesgue-typical geodesic ray enters and leaves infinitely many
horoballs in the packing. We use this setup to estimate from below the
word length along a Lebesgue-typical geodesic in terms of the sum of
the excursions in these horoballs.

We say a basepoint $x_0 \in \mathbb{H}^2$ is generic if the stabilizer
of $x_0$ in $G$ is trivial.  The $G$-orbit of the basepoint $x_0$ is
called a lattice, and if $x_0$ is a generic basepoint, then each
lattice point corresponds to a unique group element. We shall assume
that we have chosen a generic basepoint, and then each point $\gamma_t$
along the geodesic has at least one closest lattice point $h_t x_0$,
and in fact this closest point is unique for almost all points along
the geodesic.

\subsection{Projected paths are quasigeodesic}

Let us now fix some thick part $N$ of $X$, and let $\widetilde{N}$ be 
its preimage in the universal cover. The space $\widetilde{N}$
is a geodesic metric space with the following path metric. 
Every two points $x, y$ in $\widetilde{N}$ are connected by some arc, and 
the \emph{path metric} between $x$ and $y$ is defined as the infimum 
of the (hyperbolic) lengths of all rectifiable arcs connecting $x$ and $y$. 
We shall denote this distance as $d_{\widetilde{N}}(x, y)$.
Since the quotient $G\backslash\widetilde{N} = N$ is compact, then by the \v{S}varc-Milnor
lemma the space $\widetilde{N}$ with the path metric is quasi-isometric
to the group $G$ endowed with the word metric. 
A geodesic for the metric $d_{\widetilde{N}}$ will be called a \emph{thick geodesic}.

In order to have a better control on the geometry of the thick part, we shall now define a canonical 
way to connect two points in the thick part, and prove that these canonical paths (which we call 
\emph{projected paths}) are quasigeodesic for the path metric on $\widetilde{N}$.

Each point of $\mathbb{H}^2$ has a unique closest point in the thick part $\widetilde{N}$, 
hence we can define the closest point projection map $\pi_{\widetilde{N}} : \mathbb{H}^2 \to \widetilde{N}$.
Any two points $x, y$ in the thick part $\widetilde N$ are connected by a hyperbolic geodesic segment
$\gamma$ in $\mathbb{H}^2$, which may pass through a number of horoballs in $\mathcal{H}$. 
The \emph{projected path} $p(x, y)$ between $x$ and $y$ is the closest point projection of the geodesic segment 
between $x$ and $y$ to the thick part:
$$p(x, y) := \pi_{\widetilde{N}}(\gamma).$$
More explicitly, the geodesic $\gamma$ intersects a finite number $r$ (possibly zero) of horoballs of the collection $\mathcal{H}$, 
which we denote as $H_1, \dots, H_r$, and the intersection of $\gamma$ with $\widetilde{N}$ is the union 
of $r+1$ geodesic segments 
$$[x, x_1] \cup [x_2, x_3] \cup \dots \cup [x_{2r}, y].$$  
The projected path $p(x, y)$
follows the geodesic segment $[x, x_1]$ in the thick part, then follows the boundary of the horoball
$H_1$ from $x_1$ to $x_2$, then again the geodesic segment $[x_2, x_3]$ and so on, alternating
paths on the boundary of the horoballs $H_i$ with hyperbolic geodesic segments in the thick part
until it reaches $y$. 
Given $x$ and $y$ in $\widetilde{N}$, we shall denote as $L(x, y)$ the length of the projected path 
$p(x, y)$ joining $x$ and $y$.

The usefulness of projected paths arises from the fact that they are quasigeodesic, 
as proven in the following lemma.

\begin{lemma} \label{lemma:projected path qg} %
There are positive constants $L, K$ and $c$, such that if the distance
between the horoballs is at least $L$, then the projected path $p$ is
a $(K, c)$-quasigeodesic in the thick part $\widetilde N$.
\end{lemma}

\begin{proof}
Let $\gamma$ be a geodesic ray in $\mathbb{H}^2$, both of whose
endpoints lie in the thick part $\widetilde N$. Let $p$ be the
projected path, and let $q$ be the thick geodesic in $\widetilde N$
connecting the endpoints of $\gamma$.  As $q$ is a thick geodesic, the
length of $q$ is at most the length of the projected path $p$. We now
show that the length of the thick geodesic $q$ is at least the length
of the projected path $p$, minus $2n$, where $n$ is the number of
horoballs the geodesic $\gamma$ intersects. As long as the distance
between the horoballs is at least $4$, this implies that $p$ is a $(2,
2)$-quasigeodesic.

Label the intersecting horoballs $H_i$, in the order in which they
appear along $\gamma$. The hyperbolic geodesic $\gamma$ intersects the
boundary of each horoball twice, and we shall label these
intersections $\gamma_{t_{2i-1}}$ and $\gamma_{t_{2i}}$, as
illustrated below in Figure \ref{pic:chunk}.

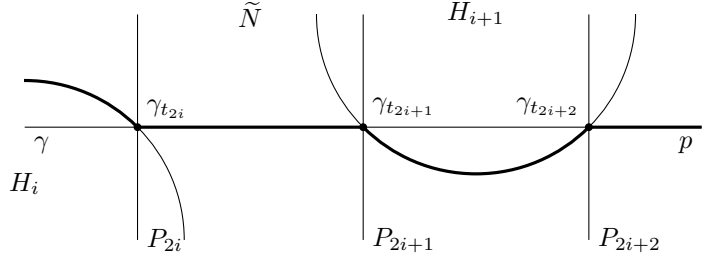
\begin{figure}[H]
\centering
\begin{tikzpicture}[scale=1.5]

\draw (0, 0) node [below right] {$\gamma$} -- (6, 0) node [below left] {$p$};

\draw (1, -1) node [right] {$P_{2i}$} -- (1, 1);
\filldraw [black] (1, 0) circle (0.03) node [above right] {$\gamma_{t_{2i}}$};

\draw (3, -1) node [right] {$P_{2i+1}$} -- (3, 1);
\filldraw [black] (3, 0) circle (0.03) node [above right] {$\gamma_{t_{2i+1}}$};

\draw (5, -1) node [right] {$P_{2i+2}$} -- (5, 1);
\filldraw [black] (5, 0) circle (0.03) node [above left] {$\gamma_{t_{2i+2}}$};

\begin{scope}
	\clip (0, -1) rectangle (6, 1);
	\draw (0, -1) circle (1.414);
	\draw (4, 1) circle (1.414);
\end{scope}

\path (0, -0.5) node {$H_i$};
\path (4, 1) node {$H_{i+1}$};

\path (2, 1) node {$\widetilde N$};

\draw [very thick] (1, 0) -- (3, 0);
\draw [very thick] (5, 0) -- (6, 0);

\begin{scope}
	\clip (0, -1) rectangle (1, 1);
	\draw [very thick] (0, -1) circle (1.414);
\end{scope}

\begin{scope}
	\clip (3, -1) rectangle (5, 1);
	\draw [very thick] (4, 1) circle (1.414);
\end{scope}

\end{tikzpicture}
\caption{Perpendicular geodesics through intersections of $\gamma$ and
  $\partial H_i$.}
\label{pic:chunk}
\end{figure}

For each point of intersection $\gamma_{t_i}$, let $P_i$ be the
perpendicular geodesic to $\gamma$ through $\gamma_{t_i}$. Each
perpendicular geodesic $P_i$ separates the endpoints of $\gamma$, so
any path connecting the endpoints must pass through each perpendicular
plane. Furthermore, the perpendicular geodesics are all disjoint, so
they divide the hyperbolic plane into regions, each of which contains
a subsegment of $\gamma$ which is either entirely contained in the
thick part $\widetilde N$, or else is entirely contained in a single
horoball. As the regions are disjoint, the length of any path is the
sum of the lengths of its intersections with each region. We now show
that the length of the thick geodesic $q$ in each region is bounded
below by the length of the projected path in that region, up to a
bounded additive error.

First consider a region between an adjacent pair $P_{2i}$ and
$P_{2i+1}$ of perpendicular geodesics containing a segment of $\gamma$
of length $d_{2i}$ in the thick part $\widetilde N$. The length of the
projected path $p$ inside this region has length exactly $d_{2i}$.  As
nearest point projection onto the geodesic is distance decreasing in
$\mathbb{H}^2$, any path from $P_{2i}$ to $P_{2i+1}$ has length at
least $d_{2i}$ in the hyperbolic metric, and hence also in the thick
metric. Therefore the intersection of the thick geodesic $q$ with this
region has length at least $d_{2i}$, i.e. at least the length of the
projected path.

Now consider a region between an adjacent pair $P_{2i+1}$ and
$P_{2i+2}$ of perpendicular geodesics containing a segment of $\gamma$
of length $d_{2i+1}$ in the boundary of a horoball $H_{i+1}$. The
length of the projected path $p$ in this region has length exactly
$d_{2i+1}$. The image of the part of the perpendicular geodesic
$P_{2i+1}$ in the thick part $\widetilde N$ projected onto the
horoball $H_{i+1}$ has diameter at most $1$. Similarly, image of the
part of the perpendicular geodesic $P_{2i+2}$ in the thick part
$\widetilde N$ projected onto the horoball $H_{i+1}$ also has
diameter at most $1$. Therefore, as the nearest point projection from
$\mathbb{H}^2 \setminus H_{i+1}$ onto the boundary of the horoball
$H_{i+1}$ is distance decreasing, the length of any path in
$\widetilde N$ between $P_{2i+1}$ and $P_{2i+2}$ has length at least
$d_{2i+1} - 2$.

This implies that the length of the thick geodesic $q$ is at least the
length of the projected path, minus $2n$, where $n$ is the number of
horoballs the geodesic $\gamma$ passes through. If we assume that the
horoballs are distance at least $L \ge 4$ apart, then the length of
the thick geodesic is at least half the length of the projected path,
up to an additive error of at most $2$.
\end{proof}

\subsection{The word metric for Fuchsian groups}

We now show that word length is coarsely monotonic along
geodesics. Recall that we write $h_t$ to denote the closest lattice
point to $\gamma_t$.

\begin{proposition} \label{prop:fuchsian word metric}
There are constants $c_1>0$ and $c_2$ such that for any geodesic
$\gamma$ and for any 
$0 \leqslant s \leqslant t$ 
\[
d_G(1, h_s) \leqslant c_1 d_G(1, h_t) + c_2.
\]
\end{proposition}

\begin{proof}
Let $p_t := \pi_{\widetilde{N}}(\gamma_t)$ be the point on the projected path that is closest to $\gamma_t$. Recall that $L(x,y)$ is the length of the projected path joining $x$ and $y$. The function $t \mapsto L(x_0, p_t)$ is continuous and for any $0 \leqslant s \leqslant t$ it satisfies $L(x_0, p_s) \leqslant L(x_0, p_t)$. The proposition then follows as the projected path is a $(K, c)$-quasi geodesic in the thick part $\widetilde{N}$, and the thick part with its path metric is quasi-isometric to $G$ with the word metric.
\end{proof}

In the Fuchsian case, we shall define the excursion of $\gamma_t$ with
respect to the horoball $H$ to be the length (in $\widetilde N$) of
the intersection of the projected path $p(0, t)$ from $p_0$ to $p_t$
with the horoball $H$, i.e.
\[ E(\gamma_t, H) := L_{\widetilde N}(p(0, t) \cap H ), \]
where $L_{\widetilde N}$ denotes the length of the path in the
$\widetilde N$-metric. We shall just write $E(\gamma, H)$, for
$\lim_{t \to \infty} E(\gamma_t, H)$, and this limit is finite for
each horoball for almost all geodesic rays. This definition of the
excursion $E(\gamma, H)$ differs from the definition in the case of
Teichm\"uller geodesics, but the two definitions are equivalent up to
additive error.

We now show that the sum of the excursions along the geodesic gives a
lower bound on the word length, using the cutoff function $\lfloor x
\rfloor_A$, as defined previously in \eqref{floor}.

\begin{proposition}  \label{prop:wm}
There are constants $A > 0, c > 0$ and $d$ such that 
\begin{equation} \label{eq:wm} %
d_G(1,h_t) \geqslant \sum_{H \in \mathcal{H}} c \lfloor E(\gamma_t, H)
\rfloor_A - d.
\end{equation}
\end{proposition}
 
\begin{proof}
The excursion $E(\gamma_t, H)$ is the length of the horocyclic segment
of the projected path in $\partial H$, and so the sums of the lengths
of the excursions is a lower bound on the length of the projected
path.  The projected path $p$ is quasi-geodesic in $\widetilde N$, and
$\widetilde N$ is quasi-isometric to the word metric, and so the
result follows.
\end{proof}

\subsection{The geodesic flow}

Let $\mathcal{H}_n$ be the subset of the horoballs $\mathcal{H}$
consisting of those points which are at least distance $\log n$ from
the boundary of the horoballs in the hyperbolic metric, i.e.

$$ \mathcal{H}_n := \{ x \in \mathbb{H}^2 \ : \ d(x, \partial
\mathcal{H}) \geqslant \log n \}. $$

Let us denote as $X_n$ the quotient of $\mathcal{H}_n$ under the
action of $G$, so $X_n \subset X$.  We will write $T^1 X$ for the unit
tangent bundle to $X$, and we will write $T^1 Y$ for the restriction
of the unit tangent bundle to any subset $Y \subset X$. Given a
geodesic ray $\gamma$, we will write $v(\gamma_t)$ for the unit
tangent vector to $\gamma$ at the point $\gamma_t$. 
Let $\ell$ denote the Liouville measure on $T^1 X$. Since the geodesic flow
 on $T^1 X$ is ergodic, for any function $\psi \in L^1(T^1 X, \ell)$, and 
for almost every geodesic ray $\gamma$, we have the equality
\[ \lim_{T \to \infty} \frac{1}{T} \int_0^T \psi(v(\gamma_t)) dt = \int_X \psi(v)
d \ell.  \]
In particular, the proportion of time that a geodesic ray spends in
$X_n$ is asymptotically the same as the volume of $T^1 X_n$, and an
elementary calculation in hyperbolic space shows that this volume is
$1/n$, up to a multiplicative constant depending on the choice of cusp
horoballs.  Let $\chi_n$ be the characteristic function of
$T^1 X_{2^n}$, and let $\psi : T^1 X \to \mathbb{R}$ be
\[
\psi(v) := \sum_{n=1}^\infty 2^n \chi_n (v).
\]
This function is not in $L^1(T^1 X, \ell)$, but it is well defined, since
each $v$ lies in finitely many $X_n$. We now show that,
as a consequence of the $1/n$ decay of volumes, the ergodic average of
$\psi$ is infinite.

\begin{proposition} \label{prop:div} %
For almost every tangent vector $v \in T^1 X$ with respect to Liouville measure, we have
\begin{equation} \label{divergence}
\lim_{T \to \infty} \frac{1}{T} \int_0^T \psi(v(\gamma_t)) d t = \infty.
\end{equation}
\end{proposition}

\begin{proof}
Let $\psi_N$ be the truncation
\[
\psi_N(v) = \sum_{n=1}^N 2^n \chi_n (v),
\]
which does lie in $L^1(T^1 X, \ell)$, and is a lower bound for $\psi$. Up to
a uniform multiplicative constant,
\[
\int_{T^1 X} \psi_N  \ d\ell \asymp N.
\]
By ergodicity, along $\ell$-almost every geodesic ray $\gamma$ 
\[
\lim_{T \to \infty} \frac{1}{T} \int_0^T \psi_N(v(\gamma_t)) d t =
\int_{T^1 X} \psi_N \ d\ell \asymp N
\]
where $v(\gamma_t)$ is the unit tangent vector to $\gamma$ at the
point $\gamma_t$.  As a consequence, along $\ell$-almost every
geodesic ray $\gamma$ the inequality
\[
\liminf_{T \to \infty} \frac{1}{T} \int_0^T \psi(v(\gamma_t)) d t
\geqslant \lim_{T \to \infty} \frac{1}{T} \int_0^T \psi_N(v(\gamma_t))
d t \asymp N
\]
holds for all $N$, which yields the claim.
\end{proof}

\begin{proposition} \label{prop:exc} %
Let $H$ be a horoball in $\mathcal{H}$, and let $t_1<t_2$ be the entry
and exit times in $H$ for a geodesic ray $\gamma$, and let $A > 0$ be
a constant. Then up to uniform additive and multiplicative constants,
which depend on $A$,
\[
\int_{t_1}^{t_2} \psi(v(\gamma_t)) \ dt \asymp \lfloor E(\gamma, H)
\rfloor_A,
\]
where $\lfloor x \rfloor_A$ is the cutoff function defined in
\eqref{floor}.
\end{proposition}

\begin{proof}
Let $N$ be the smallest number such that $\psi(v(\gamma_t)) =
\psi_N(v(\gamma_t))$ for $t \in [t_1, t_2]$, so up to a uniform
additive constant $2^{N} \leqslant E(\gamma, H) \leqslant 2^{N+1}$. We
shall write $H_n$ for the intersection of the horoball $H$ with
$\mathcal{H}_n$, so $H_n$ consists of all points of $H$ that are
distance at least $\log n$ from $\partial H$.  In the upper half-plane
model for hyperbolic space, we may assume that the boundaries of the
$H_n$ are given by horizontal lines, and the geodesic $\gamma$ is part
of a circle perpendicular to the real line. The hyperbolic distance
between $H_{2^k}$ and $H_{2^{k+1}}$ is independent of $k$, and the
shortest geodesic running between them is a vertical line, and the
longest geodesic segment is given by a semicircle tangent to the upper
horizontal line. This implies that for $k \leqslant N-1$, there are
uniform lower and upper bounds independent of $k$ and $N$ for the
amount of time $s_k$ that the geodesic ray $\gamma$ can spend in
$H_{2^k} \setminus H_{2^{k+1}}$. There is also a uniform upper bound
independent of $N$ for the amount of time $s_N$ that the ray $\gamma$
can spend in $H_{2^N} \setminus H_{2^{N+1}}$. These bounds imply
\[
\int_{t_1}^{t_2} \psi_N (v(\gamma_t))\ dt \asymp \sum_{k=1}^{N} s_k
\left(\sum_{j=1}^{k} 2^j\right) \asymp 2^N \asymp E(\gamma,H).
\]
Finally, we observe that the function $x$ is equivalent to $\lfloor x
\rfloor_A$, up to a suitably chosen additive constant, and so the
result follows.
\end{proof} 

Combining Propositions \ref{prop:wm}, \ref{prop:exc} and Equation
\eqref{divergence} we obtain the
\begin{proposition}\label{Fuchsian-lim}
For Lebesgue-almost every $\gamma$ we have 
\[
\lim_{T \to \infty} \frac{1}{T} d_G(1, h_T) = \infty.
\]
\end{proposition}
On the other hand, the relative length of $h_T$ is up to a uniform
multiplicative constant bounded above by $T$. In fact, by ergodicity,
the ray $\gamma$ spends a definite proportion of its time in the thick
part of $X$. This implies that the relative length of $h_T$ grows
linearly in $T$. Combining this observation with the limit above
proves the first part of Theorem \ref{h2}.

\subsection{Random walks} \label{section:h2random}

In this section, we prove the second part of Theorem \ref{h2}. We
start by verifying the linear progress properties that we require.
Since $G$ is non-amenable, a random walk makes linear progress in the
word metric as shown by Kesten and Day (see Theorem \ref{linword}). 
Moreover, the random walk makes linear progress in the relative
metric, too:

\begin{proposition}[Maher-Tiozzo \cite{MT}] \label{prop:linear rel fuchs}
Let $\mu$ be a probability distribution on a non-compact finite
covolume Fuchsian group $G$ which has finite first moment in the 
word metric, and
such that the semigroup generated by its support is 
a non-elementary subgroup of $G$. 
Then there is a constant $c > 0$ such that
\[ \lim_{n \to \infty} \frac{d_{rel}(1, w_n)}{n} = c. \]
\end{proposition}

The result in \cite{MT} is stated in general for random walks on (not necessarily proper) Gromov
hyperbolic spaces, and it applies here since it is well-known that the Fuchsian group $G$ 
with the relative metric is $\delta$-hyperbolic. An earlier result, under the additional 
hypothesis of convergence to the boundary and finite support, is proven in \cite{Mah2}.

\medskip

Let us now turn to the proof of Theorem \ref{h2}. 
As the random walk makes linear progress in both the word metric and
the relative metric, by taking the quotient, the limit
\[
\lim_{n \to \infty} \frac{d_G(1,w_n)}{d_{rel}(1,w_n)} 
\]
exists and is finite along almost every sample path $w = (w_1, w_2,
\dots)$.
As before, we wish to obtain a limit for points along the geodesic
$\gamma$, and so we need to relate the sample path locations $w_n x_0$
to the geodesic $\gamma$. In order to do so, we can apply exactly the 
same sublinear tracking argument of section \ref{section:distance}; 
it turns out that the Fuchsian group case is a bit easier, since it 
is not necessary to worry about the thick part. Indeed, exactly the 
same proof as in Proposition \ref{sublinear} yields the following analogue for Fuchsian groups:

\begin{proposition} \label{sublinear fuchsian}
For almost every sample path $(w_n)_{n \in \mathbb{N}}$, with corresponding geodesic ray
$\rho_w$, there exists a sequence of times $t_n \to \infty$ 
such that
\[ \lim_{n \to \infty} \frac{d_G(w_n, h_n)}{n} = 0 \]
for any $h_n \in \textup{proj}(\rho_w({t_n}))$.
\end{proposition}

Theorem \ref{h2} now follows from this Proposition using the same argument as in the proof of Theorem \ref{mcg}
in section \ref{section:intermediate}; the only thing which needs changing is that in this case we apply 
Proposition \ref{prop:fuchsian word metric} instead of Proposition \ref{coarse-mono2}. For this reason, 
since the statement of Proposition \ref{prop:fuchsian word metric} has no restriction to the thick part, 
we get for Fuchsian groups the stronger statement that the limit 
$$\lim_{t \to \infty} \frac{d_G(1, h_t)}{d_{rel}(1, h_t)}$$ 
exists for $\nu$-almost every geodesic without any restriction to subsequences, completing the proof of Theorem \ref{h2}.

\section{Lyapunov expansion exponent}\label{section:lyap}

We consider the Lyapunov expansion exponent defined by Deroin-Kleptsyn-Navas in \cite{Der-Kle-Nav}. 
For a Fuchsian group $G$, let $d_G$ be the word metric with respect to a finite set of generators. 
Let $B(R)$ denote the ball of radius $R$ in $G$ for the word metric $d_G$. Given a point $ \in S^1$, the {\em Lyapunov expansion exponent} at $p$ is defined as:
\[
\lambda_{exp}(p) = \limsup_{R \to \infty} \max_{g \in B(R)}  \frac{1}{R} \log \vert g'(p) \vert.
\]
As an application of Theorem \ref{h2}, we prove:
\begin{theorem}\label{Lyap}
If $G$ is a Fuchsian group with parabolic elements, then for Lebesgue-almost every $p \in S^1$
the Lyapunov expansion exponent is zero:
\[
 \lambda_{exp}(p) = 0.
\]
\end{theorem}
\noindent This answers Question 3.3 in \cite{Der-Kle-Nav} in the affirmative. 

Here is the rough idea of the proof of Theorem \ref{Lyap}. Suppose $p$
is a point in $S^1$ and let $\gamma$ be the hyperbolic geodesic ray
that connects the origin $x_0$ in $\mathbb{D}$ to $p$. Let $h_T$ be
the approximating group element for $\gamma_T$. We will show that for
every group element in a ball of radius $R = d_G(1, h_T)/2K^2$ where
$K$ is some uniform constant, the derivative at $p$ has a coarse upper
bound of $e^{2T}$. As $T$ increases, the word length of the
approximating group elements is monotonically increasing with bounded
jump size. Finally, for Lebesgue-almost every $p$, Proposition
\ref{Fuchsian-lim} says that the ratio $T/R$ goes to zero, which
proves Theorem \ref{Lyap}.
 
\subsection{Derivatives of isometries}

We shall use the unit disc model $\mathbb{D}$ of hyperbolic plane. An isometry of $\mathbb{D}$ is of the form 
\[
f(z) = e^{i \theta} \frac{z-a}{1-\overline{a}z}
\]
where $a \in \mathbb{D}$. Write $a$ as $a= Ae^{i \phi}$ and suppose
$f(e^{it}) = e^{i g(t)}$. Differentiation with respect to $t$, and an
elementary calculation, shows that
\begin{equation} \label{eq:der}
\vert g'(t) \vert = \frac{1-A^2}{1+A^2- 2A \ \textup{Re}(e^{i\phi}e^{-it})}.
\end{equation}
It follows that $\vert g'(t) \vert$ is maximum with value $(1+A)/(1-A)$ when $t =  \phi$. Denoting the origin in $\mathbb{D}$ as $x_0$, note that $(1+A)/(1-A) = e^{d_{\mathbb{H}^2}(x_0, f(x_0))}$
 and so in particular, the calculation shows that the maximum value of the logarithm of the derivative on $S^1$ is equal to the hyperbolic distance that $f$ moves the origin $x_0$. To summarize, we get

\begin{lemma}\label{bounded-der}
If $g$ is an isometry of $\mathbb{D}$ such that $d_{\mathbb{H}^2}(x_0, gx_0) \leqslant T$ then for any $p \in S^1$,
\[
\vert g'(p) \vert \leqslant e^T.
\]
\end{lemma}

\subsection{Bounding derivative over a ball in the word metric }

Let $p \in S^1$, and $\gamma$ be the geodesic ray from the origin $x_0$ to $p$. Let 
$p_T = \pi_{\widetilde{N}}(\gamma_T)$ denote the point in the thick part closest to $\gamma_T$ 
and let $h_T$ be the approximating group element. Let
\[
H(x_0, \gamma_{2T}) = \{ x \in \mathbb{D} : d_{\mathbb{H}^2}(x_0, x) \geqslant d_{\mathbb{H}^2}(\gamma_{2T}, x) \}.
\]
Thus, $H(x_0, \gamma_{2T})$ is the half-space with $\partial H(x_0, \gamma_{2T})$ orthogonal to $\gamma$ at the point $\gamma_{T}$. 

\begin{proposition}\label{shadow}
There exists constants $K$, $K'$ such that, if $gx_0$ lies in $H(x_0, \gamma_{2T})$, then 
\[
d_G(1, g) \geqslant \frac{1}{K} d_G(1, h_T) - K'.
\]
\end{proposition}

Before proving Proposition \ref{shadow}, we state a basic lemma in hyperbolic geometry. 
If $H$ is a horoball, we shall denote as $\pi_H$ the closest point projection map 
onto the boundary of $H$; moreover, if $x, y$ lie on $\partial H$, we denote 
as $d_{\partial H}(x, y)$ the length of the path along the boundary of $H$ between $x$ and $y$.

\begin{lemma}\label{basic} 
Fix a point $y \in \mathbb{D}$ and let $H$ be a horoball that does not contain $y$. 
Let $\gamma_0$ be the hyperbolic geodesic that goes from $y$ to the point at infinity of $H$.
 Let $\pi_H(y)$ denote the point of entry of $\gamma_0$ into $H$. 
Let $\gamma$ be any geodesic ray from $y$ that enters $H$, and let $\gamma_u$ be its point of entry. Then 

\[d_{\partial H}(\gamma_u, \pi_H(y)) \leqslant 1.\]
\end{lemma}

\begin{proof}[Proof of Proposition \ref{shadow}]
Let $x = g x_0$ and let $\delta$ be the hyperbolicity constant for the hyperbolic metric $d_{\mathbb{H}^2}$. 

\vskip 10pt
\noindent {\em Case 1:} Suppose $\gamma_T$ is in the thick part. The hyperbolic geodesic from $x_0$ to $x$ must pass 
through a $3\delta$ neighborhood of $\gamma_T$ (See Proposition 3.2 of \cite{Mah2}). This means that there is a point $x'$
 on the hyperbolic geodesic from $x_0$ to $x$ that also lies in the thick part. So the projected path from $x_0$ to $x$ 
necessarily passes through $x'$. 
Recall $L(y,y')$ is the distance along the projected path between the points $y, y'$. 
It follows that 
$$L(x_0, x) = L(x_0, x') + L(x', x) \geqslant L(x_0, x')$$
hence passing to the word metric we get 
$$d_G(1, g) \asymp L(x_0, x) \geqslant L(x_0, x') \asymp d_G(1, h_T).$$

\noindent {\em Case 2:} Suppose $\gamma_T$ is in some horoball $H$ and let $\gamma_u$ and $\gamma_v$ be the points where
$\gamma$ enters and leaves $H$. We may assume that a ball of hyperbolic radius $3\delta$ about $\gamma_T$ is contained in $H$.
 Then the hyperbolic geodesic $\gamma'$ from $x_0$ to $x$ must enter and leave $H$. Denote its entry and exit points 
by $\gamma'_r$ and $\gamma'_s$. Moreover, let $p_T = \pi_{\widetilde{N}}(\gamma_T)$ be the projection 
of $\gamma_T$ to the boundary of the horoball, and denote by 
$E:= d_{\widetilde{N}}(\gamma_u, \gamma_v)$ the excursion of $\gamma$ in $H$, and $D := d_{\widetilde{N}}(\gamma_u, p_T)$.  

There are two sub-cases to consider.

\noindent {\em Case 2a:} If $D \geqslant E/2$, then we are in 
the situation of Figure \ref{case1} and $x$ must lie in the shaded region.

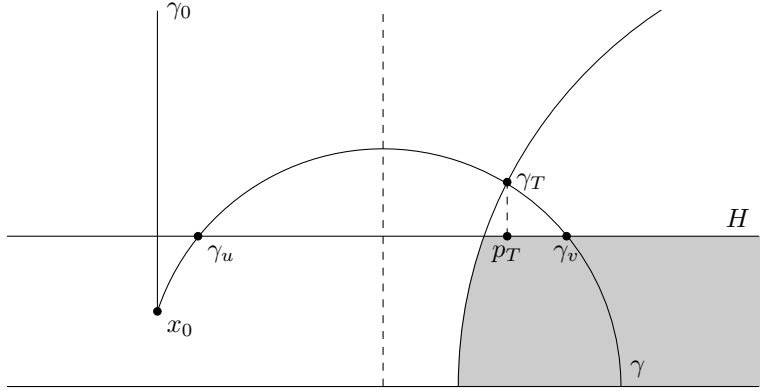
\begin{figure}[H]
\centering
\begin{tikzpicture}[scale=2]

\begin{scope}
	\clip (0, 0) rectangle (5, 1);
	\filldraw [black!20] (6, 0) +(180:3) arc (180:0:3);
\end{scope}

\draw (0, 0) -- (5, 0);
\draw (0, 1) -- (5, 1) node [above left] {$H$};

\draw (1, 0.5) node [below right] {$x_0$} -- (1, 2.5) node [right] {$\gamma_0$};

\filldraw [black] (1, 0.5) circle (0.025);

\draw (2.5, 0) +(161.6:1.5811) arc (161.6:0:1.5811) node [above right] {$\gamma$};

\draw [dashed] (2.5, 0) -- (2.5, 2.5);

\begin{scope}
	\clip (0, 0) rectangle (5, 2.5);
	\draw (6, 0) +(180:3) arc (180:0:3);
\end{scope}

\filldraw [black] (3.325, 1.36) circle (0.025) node [right] {$\gamma_T$};
\filldraw [black] (3.325, 1) circle (0.025) node [below] {$p_T$};
\filldraw [black] (1.27, 1) circle (0.025) node [below right] {$\gamma_u$};
\filldraw [black] (3.72, 1) circle (0.025) node [below] {$\gamma_v$};

\draw [dashed] (3.325, 1) -- (3.325, 1.36);

\end{tikzpicture}
\caption{Perpendicular geodesics through intersections of $\gamma$ and
  $\partial H$.}
\label{case1}
\end{figure}

In this case, let $\pi_H(x)$ be the closest points projection of $x$ onto the boundary of $H$; 
then by  
Lemma \ref{basic}, the entry point $\gamma'_r$ is within distance $1$ of $\gamma_u$ and the exit point $\gamma'_s$ is 
within distance $1$ of $\pi_H(x)$. So we get 
\[
d_{\partial H}(\gamma'_r, \gamma'_s)  \geqslant d_{\partial H}(\gamma_u, \pi_H(x)) - 2 \geqslant \frac{E}{2} - 2.
\]
On the other hand, $d_{\partial H}(\gamma_u, p_T) \leqslant E$, so we have 
\[
d_{\partial H}(\gamma'_r, \gamma'_s) \geqslant \frac{1}{2}d_{\partial H}(\gamma_u, p_T) - 2.
\]
Moreover, by Lemma \ref{basic} and Lemma \ref{lemma:projected path qg}, 
$$L(x_0, \gamma_r') \asymp d_{\widetilde{N}}(x_0, \gamma_r') \geqslant d_{\widetilde{N}}(x_0, \gamma_u) - 1 \asymp L(x_0, \gamma_u).$$
Consequently, the distances along respective projected paths satisfy
\begin{eqnarray*}
L(x_0, x) &\geqslant& L(x_0, \gamma'_r) + d_{\partial H}(\gamma'_r, \gamma'_s)\\
&\gtrsim& L(x_0, \gamma_u) + d_{\partial H}(\gamma_u, p_T)\\
&=& L(x_0, p_T).
\end{eqnarray*}
Thus, passing to the word metric we get
\[
d_G(1, g) \asymp L(x_0, x) \gtrsim L(x_0, p_T) \asymp d_G(1, h_T).
\]

\noindent {\em Case 2b:} If $D \leqslant E/2$, then we are in the situation of Figure \ref{case2} and $x$ must lie in the shaded regions.

\begin{figure}[H]
\centering
\begin{tikzpicture}[scale=1.7]

\begin{scope} [even odd rule]
    \clip (-3, 0) rectangle (4.5, 1)
    	  (0, 0) +(180:2) arc (180:0:2);
    \filldraw [black!20] (-3, 0) rectangle (4.5, 1);
\end{scope}

\draw (-3, 0) -- (4.5, 0);
\draw (-3, 1) -- (4.5, 1) node [above left] {$H$};

\draw (1, 0.5) node [below right] {$x_0$} -- (1, 2.5) node [right] {$\gamma_0$};

\filldraw [black] (1, 0.5) circle (0.025);

\draw (2.5, 0) +(161.6:1.5811) arc (161.6:0:1.5811) node [above right] {$\gamma$};

\draw [dashed] (2.5, 0) -- (2.5, 2.5);

\begin{scope}
	\clip (-3, 0) rectangle (4.5, 2.5);
	\draw (0, 0) +(180:2) arc (180:0:2);
\end{scope}

\filldraw [black] (1.55, 1.265) circle (0.025) node [above] {$\gamma_T$};
\filldraw [black] (1.55, 1) circle (0.025) node [below] {$p_T$};

\filldraw [black] (1, 1) circle (0.025) node [below left] {$p_0$};
\filldraw [black] (1.28, 1) circle (0.025) node [above] {$\gamma_u$};
\filldraw [black] (0.35, 1) circle (0.025) node [below] {$p'_T$};

\draw [dashed] (1.55,1) -- (1.55, 1.265);

\end{tikzpicture}
\caption{Perpendicular geodesics through intersections of $\gamma$ and
  $\partial H$.}
\label{case2}
\end{figure}
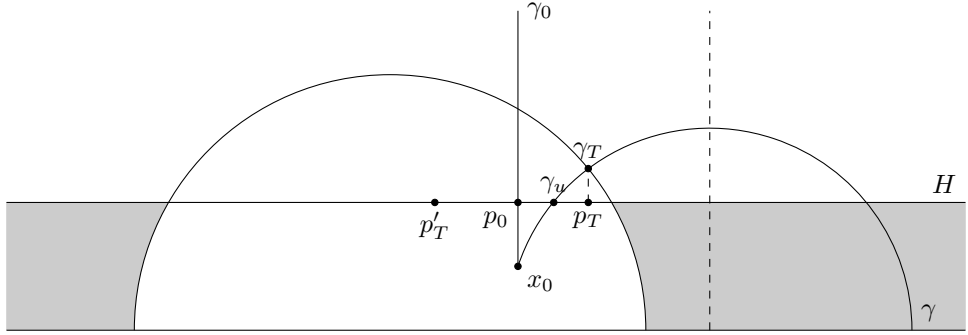

If $x$ is in the shaded region on the right, then note that 
$$d_{\partial H}(\gamma_u, \pi_H(x)) \geqslant d_{\partial H}(\gamma_u, p_T),$$
 which by Lemma \ref{basic} implies 
$$d_{\partial H}(\gamma'_r, \gamma'_s) \geqslant d_{\partial H}(\gamma_u, p_T) - 2,$$
and the required estimate for $d_G(1,g)$ then follows by estimates on distances along 
respective projected paths similar to {\em Case 2a}. 
If $x$ is in the shaded region on the left, let $p'_T$ be the point on $\partial H$ such that $p_T$ and $p'_T$ are symmetric about $\gamma_0$,
 the geodesic ray from $x_0$ to the point at infinity for $H$, and denote $p_0 = \pi_H(x_0)$. Observe that 
$d_{\partial H}(p_0, \pi_H(x)) \geqslant d_{\partial H}(p_0, p'_T)$. Hence, by Lemma \ref{basic},
\begin{eqnarray*}
d_{\partial H}(\gamma'_r, \gamma'_s) &\geqslant& d_{\partial H}(p_0, \pi_H(x)) - 2\\
&\geqslant& d_{\partial H}(p_0, p'_T) - 2\\
&=& d_{\partial H}(p_0, p_T) - 2 \\
&\geqslant& d_{\partial H}(\gamma_u, p_T) - 3, 
\end{eqnarray*}
and the required estimate for $d_G(1,g)$ then follows by estimates on distances along respective projected paths similar to {\em Case 2a}.

\end{proof}

Let $K$, $K'$ be the constants in Proposition \ref{shadow}, and for each $T$ let
$$R_T := d_G(1, h_T)/K - K'.$$ 
Consider the ball $B(R_T)$ of radius $R_T$ in $G$ in the word metric; our goal is to prove an upper bound 
on the derivatives of the elements in the ball. Let us first establish another elementary lemma in hyperbolic geometry.

\begin{lemma}\label{angle}
Let $L > T$ a constant, and $y_1$ and $y_2$ be points on $\partial H(x_0, \gamma_{2T})$ such that $d_{\mathbb{H}^2}(y_i, \gamma_u) = L-T$.
The points $y_1$ and $y_2$ are symmetric about $\gamma$, and let $\psi$ be the angle between the ray $\gamma'$ from 
$x_0$ through $y_1$ and the original ray $\gamma$. 
Then there exists a constant $C >0$ such that, if $T$ is sufficiently large (say when $\tanh T > 1/2$) and 
$L \geqslant 2T$, then  
\[
\psi \geqslant C e^{-T}.
\]
\end{lemma}
\begin{proof}
By a hyperbolic trigonometric identity for the right triangle $\Delta(x_0, \gamma_T, y_1)$ we have
\[
\tan \psi = \frac{\tanh (L-T)}{\sinh T} = \frac{2e^T}{e^{2T}-1} \cdot \frac{e^{2(L-T)} -1}{e^{2(L-T)}+1}.
\]
If $L>2T$ and $T$ is large enough, then the second fraction on the right hand side is at least $1/2$. 
Also with $T$ large enough (greater than a uniform threshold) the approximation $\tan \psi \asymp \psi$ is 
true up to a fixed multiplicative constant that depends only on the threshold. This proves the lemma.
\end{proof}

\begin{proposition} \label{der-estimate}
Any $g \in B(R_T)$ satisfies
\[
\vert g'(p) \vert \lesssim e^{2T}
\]
for each $p \in S^1$.
\end{proposition}

\begin{proof}
Fix $p \in S^1$, and let $\gamma$ be the geodesic ray from the origin $x_0$ of the unit disc to $p$.
Fix $T > 0$ and $g \in B(R_T)$, and let $L := d_{\mathbb{H}^2}(x_0, gx_0)$.
By Lemma \ref{bounded-der}, if $L \leqslant 2T$, then $\vert g'(p)
\vert \leqslant e^{2T}$ which implies the proposition. Hence, we may
assume $L \geqslant 2T$. 
Let $y_1$ and $y_2$ be points on $\partial H(x_0, \gamma_{2T})$ such that $d_{\mathbb{H}^2}(y_i, \gamma_u) = L-T$, and 
let $U$ be the sector subtended at $x_0$ by rays from $x_0$ passing through $y_1$ and $y_2$. We claim that the point
$g x_0$ cannot be in $U$. 
Indeed:
\begin{itemize}
 \item the point $g x_0$ cannot lie in $H(x_0, \gamma_{2T})$, because otherwise
(by Proposition \ref{shadow} and the definition of $R_T$) the word length of $g$ satisfies $d_G(1, g) > R_T$, 
contradicting the fact that $g$ is in $B(R_T)$;
\item  $g x_0$ cannot lie in $U \setminus H(x_0, \gamma_{2T})$, because otherwise it belongs to the geodesic triangle 
$\Delta(x_0, y_1, y_2)$, hence $d_{\mathbb{H}^2}(x_0, gx_0) < (L-T) + T = L$.
\end{itemize}

Now, by the derivative calculations (equation \eqref{eq:der})
\[
\vert g'(p) \vert = \frac{1-A^2}{1+ A^2 - 2A \cos \phi}
\]
where $\phi$ is the angle between $\gamma$ and the geodesic ray joining $x_0$ with $g x_0$, and $A = (e^{L} -1)/(e^{L}+1)$.
By Lemma \ref{angle}, the angle $\phi$ satisfies $\phi \geqslant \psi \geqslant C e^{-T}$. Hence,
\begin{eqnarray*}
\vert g'(p) \vert &\leqslant& \frac{4e^{L}}{2e^{2L}(1- \cos \psi) + 2(1+ \cos \psi)}\\
&\leqslant& \frac{e^{-L}}{\sin^2 (\psi/2) }\\
&\asymp& e^{2T-L}\\
&\leqslant& 1 \leqslant e^{2T}
\end{eqnarray*} 
where the second to last inequality follows from the assumption $L \geqslant 2T$. This proves the proposition.
\end{proof}

\subsection{Proof of Theorem \ref{Lyap}}

Before proving the theorem, we still need to 
 show that the function $T \to d_G(1, h_T)$ has bounded jump size, in the following sense. 

\begin{lemma}\label{jump-size}
For $\gamma$ a hyperbolic geodesic ray, let us define the set
$$\mathcal{R}(\gamma) := \{ r \in \mathbb{Z}_{\geqslant 0} : r = d_G(1, h_T) \text{ for some } T\}.$$ 
If $\gamma$ is recurrent to the thick part, then the set $\mathcal{R}(\gamma)$ is infinite, 
and we can index its elements in increasing order $r_1 < r_2 < \dots$.
Then there exists a constant $k > 0$ such that for any recurrent geodesic ray $\gamma$ and any $i$, we have 
$$r_{i+1} - r_i < k.$$ 
\end{lemma}
\begin{proof}[Proof of Lemma \ref{jump-size}]
For a geodesic ray $\gamma$, recall that $p_T = \pi_{\widetilde{N}}(\gamma_T)$ is the point on the projected path 
of $\gamma$ that is the closest to $\gamma_T$. By Lemma \ref{lemma:projected path qg}, the image of the function 
$T \mapsto p_T$ is a continuous path which is $(K,c)$-quasigeodesic in $\widetilde{N}$.
Let us choose times $T_n$ along the geodesic such that $L(x_0, p_{T_n}) = n$.
Since the thick part $\widetilde{N}$ is quasi-isometric to the group $G$, then, 
up to multiplicative constants which depend only on the quasi-isometry constants, we have
$$|d_G(1, h_{T_{n+1}}) - d_G(1, h_{T_n})| \lesssim d_{\widetilde{N}}(p_{T_n}, p_{T_{n+1}}) \lesssim 1.$$
\end{proof}

Let us now turn to the proof of Theorem \ref{Lyap}. Recall that the Lyapunov expansion exponent is defined as
\[
\lambda_{exp}(p) = \limsup_{R \to \infty} \max_{g \in B(R)} \frac{1}{R} \log \vert g'(p) \vert.
\]
Lemma \ref{jump-size} implies that along geodesic rays recurrent to the thick part the corresponding values of $R$ 
given by $R_T= d_G(1, h_T)/K - K'$ are infinite and have a bounded jump size. 
So the $\limsup$ in the above definition can be replaced by a $\limsup$ over values given by $R_T$. 
By Proposition \ref{der-estimate}, for almost every $p \in S^1$, 
\[
\max_{g \in B(R_T)} \frac{1}{R_T} \log \vert g'(p) \vert \leqslant \frac{1}{R_T} \log (e^{2T}) = \frac{2T}{R_T} \asymp \frac{2T}{d_G(1, h_T)}.
\]
Hence, by Proposition \ref{Fuchsian-lim} for Lebesgue-almost every $p$
\[
\lambda_{exp}(p) = 0
\]
proving Theorem \ref{Lyap}.


\vskip 20pt
\noindent Vaibhav Gadre \\
Warwick Mathematics Institute \\
Zeeman Building, University of Warwick, Coventry CV4 7AL, UK\\
\url{v.gadre@warwick.ac.uk; gadre.vaibhav@gmail.com} \\

\noindent Joseph Maher \\
CUNY College of Staten Island and CUNY Graduate Center \\
2800 Victory Boulevard, Staten Island NY 10314 USA \\
\url{joseph.maher@csi.cuny.edu} \\

\noindent Giulio Tiozzo \\
Yale University \\
10 Hillhouse Avenue, New Haven CT 06511 USA \\
\url{giulio.tiozzo@yale.edu}


\end{document}